\documentclass{tran-l}

\usepackage{fancyhdr}
\usepackage{color}
\usepackage{ulem}
\newtheorem{theorem}{Theorem}[section]
 \newtheorem{lemma}[theorem]{Lemma}
\newtheorem{remark}[theorem]{Remark}

\numberwithin{equation}{section}
\newcommand{\eps}{\varepsilon}
\newcommand{\sgn}{\text{sgn}}
\newcommand{\R}{\mathbb{R}}

\begin{document}


%
%


\title{Multiscale Analysis of a Kinetic Model of \\ Confined Suspensions of Self-Propelled Rods}

\author{Leonid Berlyand}

\address{Department of Mathematics, The Pennsylvania State University,\\
University Park, Pennsylvania, 16802, USA\\
lvb2@psu.edu}

\author{Spencer Dang}

\address{Department of Mathematics, The Pennsylvania State University,\\
University Park, Pennsylvania, 16802, USA\\
sbd5473@psu.edu}

\author{Pierre-Emmanuel Jabin}

\address{Department of Mathematics, The Pennsylvania State University,\\
University Park, Pennsylvania, 16802, USA\\
pejabin@psu.edu}

\author{Mykhailo Potomkin}
\address{%
 Department of Mathematics, University of California, Riverside,\\
 900 University Ave.,
Riverside, California, 92521, USA\\
mykhailp@ucr.edu
}%
\maketitle


\begin{abstract}

The behavior of active matter under confinement poses significant challenges due to the intricate coupling between dynamics near boundaries and those in the bulk. A defining feature of active matter systems is that a substantial portion of their dynamics takes place near confining boundaries. In our previous work, we developed a kinetic framework that enables direct computation of the probability distribution functions for both the position and orientation of active rods. A distinguishing aspect of this approach is its explicit treatment of wall accumulation through the use of two coupled probability distribution functions: one describing the bulk population and the other representing rods accumulated at the boundary. Another novel feature is the structure of the governing equation, which is degenerate: it is second-order in one non-temporal variable and first-order in another.
The main focus of this paper is to rigorously justify this model via multi-scale analysis. We first establish well-posedness of the system and then employ two distinct multi-scale derivations to obtain the model as a singular limit of a more classical kinetic system in the regime of vanishing translational diffusion. For analytical clarity, we consider the case in which active rods, once accumulated at the wall, remain permanently confined there.
This work provides a rigorous mathematical foundation for reduced kinetic models of confined active matter, bridging microscopic dynamics and macroscopic accumulation phenomena.
     
\end{abstract}

\bigskip 








\date{}



\maketitle


\bibliographystyle{unsrt}

\section{Introduction} 

The concept of a self-propelled rod, also called an active rod, was successfully applied to study dynamics of a wide variety of microswimmers, ranging from motile rod-shaped bacteria \cite{SaiShe2008,peruani2006nonequilibrium,Elg2009,Bar2020} to chemically-driven bimetallic microrods \cite{rheotaxis2018baker,rubio2021}. Swimming of microorganisms typically occurs in confined environments and a common feature of rod-shaped microswimmer dynamics is accumulation at the confinements~\cite{rothschild1963non,Bia2017,BerLau2008,LiTan2009,Fig2020} leading to microswimmer's distributions with large values at walls. Unlike sedimentation of passive particles, accumulated active microswimmers continue to swim along the confining wall. The overall transport properties of dilute active microswimmers through microfluidic channels or microscale biological conduits are typically governed by the dynamics of the accumulated population~\cite{fily2014dynamics,fily2015dynamics,ezhilan2015transport,brown2024boundary}. Therefore, an accurate model that captures both microswimmer accumulation and the behavior of the accumulated population is essential for various biological and biomimetic applications.

An active rod is a thin segment with a front and rear, capable of autonomous motion toward its front. A basic model for an individual active rod swimming in a planar domain $\Omega\subset \mathbb R^2$ is governed by equations for the rod's center location $\boldsymbol{r}(t)$ and orientation angle $\varphi(t)$, defined as the counterclockwise angle between a line parallel to the rod's orientation and $x$-axis.  
If expressions for its translational velocity $\boldsymbol{V}(\boldsymbol{r},\boldsymbol{p})$ and angular velocity $\boldsymbol{\Phi}(\boldsymbol{r},\boldsymbol{p})$ are given, then the stochastic differential equations governing the active rod's dynamics are: 
\begin{equation}
    \begin{cases}
        \text{d}\boldsymbol{r}=\boldsymbol{V}(\boldsymbol{r},\boldsymbol{\varphi})\text{d}t+\sqrt{2D_{\text{tr}}}\,\text{d}W_{\boldsymbol{r}},\\
    \text{d}{\boldsymbol{\varphi}}=\Phi(\boldsymbol{r},\boldsymbol{\varphi})\text{d}t+\sqrt{2D_{\text{rot}}}\,\text{d}W_{\boldsymbol{\varphi}}.
    \end{cases}
    \label{eqn:original-sde}
\end{equation}
Here, $D_{\text{tr}}$ and $D_{\text{rot}}$ are translational and orientational diffusion coefficients, respectively, and $W_{\boldsymbol{r}}, W_{\boldsymbol{\varphi}}$ are Wiener processes. An 
example for $\boldsymbol{V}$ and $\Phi$ is\cite{potomkin2017,brown2024boundary} 
\begin{equation}
\boldsymbol{V}=V_{\text{prop}}\boldsymbol{p}+\boldsymbol{u}_{BG}+\boldsymbol{F}_{\text{ext}} \text{ and } \Phi = [\text{I}-\boldsymbol{p}\boldsymbol{p}^{\text{T}}]\nabla \boldsymbol{u}_{\text{BG}} \boldsymbol{p}\cdot \boldsymbol{e}_\varphi. 
\label{eqn:an-example}
\end{equation}
Here, $\boldsymbol{p}=(\cos(\varphi),\sin(\varphi))^{\text{T}}$ is the rod's orientation vector, $V_{\text{prop}}$ is the self-propulsion speed, $\boldsymbol{u}_{\text{BG}}$ is the background flow, and $\boldsymbol{F}_{\text{ext}}$ stands for an external force. If no external force is applied, $\boldsymbol{F}_{\text{ext}}=\boldsymbol{0}$, then the velocity of the active rod relative to the background flow is the propulsion velocity: $\boldsymbol{V}-\boldsymbol{u}_{BG} = V_{\text{prop}}\boldsymbol{p}$.  

Due to randomness coming from translational and rotational diffusion as well as uncertainty in initial data, it is convenient to translate the system \eqref{eqn:original-sde} into the Fokker-Planck equation for the probability distribution function $f(t,\boldsymbol{r},\varphi)$: 
\begin{equation}
\partial_t f + \partial_\varphi(\Phi f) - D_{\text{rot}}\partial_\varphi^2 f + \nabla_{\boldsymbol{r}}\cdot (\boldsymbol{V} f) - D_{\text{tr}} \Delta_{\boldsymbol{r}} f =0.\label{eq:unscaled_main_fp}
\end{equation}
This equation for the distribution of the locations and orientations of rod-like particles resembles the Smoluchowski equation~\cite{doi1981molecular,doi1988theory} for passive rods and has been used to analyze both dilute and semi-dilute active suspensions\cite{SaiShe2008,saintillan2008instabilities,haines2009three,saintillan2013active,ryan2013kinetic,ryan2013correlation,potomkin2016effective}.  

Translational diffusion in microswimmer's dynamics is typically negligible. To illustrate this using the Fokker-Planck equation \eqref{eq:unscaled_main_fp}, introduce non-dimensional variables $t=\hat{t}T$ and $\boldsymbol{r} = \hat{\boldsymbol{r}}L$ where $T = D_{\text{rot}}^{-1}$ and $L$ is the length associated with the domain of active rod's swimming $\Omega$. For example, if $\Omega$ is bounded, then $L =\text{diam}(\Omega)$, and if $\Omega$ is a slab of thickness $H$, $\Omega = \{(x,y)|0<y<H\}$, then $L=H$. After rescaling, we get the non-dimensional Fokker-Planck equation:
\begin{equation}
\partial_{\hat{t}} f_\varepsilon + \partial_\varphi(\hat{\Phi} f_\varepsilon) - \partial_\varphi^2 f_\varepsilon + \nabla_{\hat{\boldsymbol{r}}}\cdot (\hat{\boldsymbol{V}} f_\varepsilon) - \varepsilon\Delta_{\hat{\boldsymbol{r}}} f_\varepsilon =0.\label{eqn:rescaled_main_fp}
\end{equation}
Here, $\hat{\Phi}=\Phi D^{-1}_{\text{rot}}$, $\hat{\boldsymbol{V}}=\boldsymbol{V}/(L D_{\text{rot}})$, and 
\begin{equation}\label{eq:ratio}
\varepsilon=\dfrac{D_{\text{tr}}}{D_{\text{rot}}L^2}\textcolor{black}{=\underbrace{\dfrac{D_{\text{tr}}}{D_{\text{rot}}r^2}}_{\sim 1}\cdot\underbrace{\dfrac{r^2}{L^2}}_{\ll 1}}.
\end{equation}
Here, $r$ is the linear size of the microswimmer. For various types of microswimmers, the coefficient $\varepsilon$ is small, as indicated in \eqref{eq:ratio}, and thus the translational diffusion term is typically neglected in the modeling of active rod trajectories by the system \eqref{eqn:original-sde}, {\it e.g.} in Refs.~\cite{potomkin2017,rubio2021}. For example, for a flagellated bacterium, values of translation diffusion may attain $D_{\text{tr}}=2 \times 10^{3}\mu \text{m}^2 \cdot \text{s}^{-1}$ \cite{berg1993} whereas a representative value for the rotational diffusion coefficient is $D_{\text{rot}}=3.0 \,\text{rad}^2\cdot \text{s}^{-1}$ \cite{perron2012}.  For these values and $L = 1 \,\text{mm}$, the geometric scale is  $\varepsilon=6.6\times 10^{-4}\ll 1$.  

{This work focuses on the limit $\varepsilon \to 0$  in the equation \eqref{eqn:rescaled_main_fp}.
Equation~\eqref{eqn:rescaled_main_fp} is subject to periodic boundary condition in $\varphi$ and the no-flux (or Robin) boundary condition at the boundary }$\partial\Omega$: 
\begin{equation} 
(\hat{\boldsymbol{V}}f_\varepsilon-\varepsilon \nabla_{\hat{\boldsymbol{r}}}f_\varepsilon)\cdot \boldsymbol{n}=0, \quad \boldsymbol{r}\in\partial \Omega,\,-\pi<\varphi\leq \pi. \label{eqn:bc}
\end{equation}
The long-time behavior of individual active Brownian particle transport in confined flows based on the system of the kinetic equation~\eqref{eqn:rescaled_main_fp} with the boundary conditions~\eqref{eqn:bc} or their variations has been investigated, both analytically and numerically, in a number of studies~\cite{ezhilan2015transport,perthame2020,jiang2019dispersion,jiang2021transient,PerFu2023}. It was proven in Ref.~\cite{berlyand2020kinetic} that the boundary-value problem (\ref{eqn:rescaled_main_fp}, \ref{eqn:bc}) can be rigorously obtained in the small inertia limit for a wide range of rules describing collisions of active rods with the wall. 

{The boundary condition \eqref{eqn:bc} makes the limiting distribution $f$ of the distribution function $f_\varepsilon$ singular in the limit $\varepsilon\to 0$.} For small $\varepsilon$, a boundary layer is formed along $\partial\Omega$, which is in agreement with the wall accumulation observed in numerical simulations for system \eqref{eqn:original-sde} (see, e.g. Refs.~\cite{potomkin2017,rheotaxis2018baker}) and experiments with biological microswimmers \cite{rothschild1963non,Bia2017,BerLau2008,LiTan2009}. 

It was conjectured in Ref.~\cite{berlyand2020kinetic}, based on the formal multiscale asymptotic expansions, that the probability distribution function admits the following representation in the limit
$\varepsilon\to 0$: 
\begin{equation}
f\to \rho_{\text{bulk}} + \delta_{\partial\Omega}(\hat{\boldsymbol{r}}) \rho_{\text{wall}},\label{lim-representation}
\end{equation}
where $\rho_{\text{bulk}}$ and $\rho_{\text{wall}}$ are distributions of rods that swim freely in the bulk and rods that are accumulated on the wall $\partial \Omega$, respectively. The term $\delta_{\partial \Omega}$ denotes the $\delta$-function along $\partial \Omega$. The distributions $\rho_{\text{bulk}}$ and $\rho_{\text{wall}}$ satisfy a coupled system of Fokker-Planck equations. The main advantage of the system for $\rho_{\text{bulk}}$ and $\rho_{\text{wall}}$ is that it explicitly distinguishes accumulated rods and it does not have large variations of density values near the wall. A similar model of coupled in-bulk and wall distributions for active Brownian particles without background flow was formulated in Ref.~\cite{perthame2020} and their large-time behavior was analyzed.

The primary contribution of this work is the rigorous justification of the coupled bulk-boundary dynamics in a kinetic model of active rods in the singular limit $\eps \to 0$.  The key challenge in proving convergence is the singular limit; naive asymptotic expansions of bulk-boundary dynamics leads to an ill-posed system.  Instead, a two-scale matched asymptotic expansion was used along with a matching condition which captures the accumulation of rods on the wall (e.g., bordertaxis) to rigorously derive the limiting system~\eqref{eq:main_degenerate}. 

To this end, we start with the classical matching asymptotic expansion method~\cite{van1975perturbation,holmes2012introduction,nayfeh2024perturbation} resulting in the formal series accounting for the at-wall and away-from-wall (so called inner and outer solutions, respectively, with an appropriate matching condition). From the composite expansion $\hat f_\eps$,  rigorous $L^1$ weak convergence $\hat f_\eps \to  f = \rho_{\text{bulk}} + \delta_{\partial\Omega}(\hat{\boldsymbol{r}}) \rho_{\text{wall}}$ is proven. 
Our result requires establishing the well-posedness of the limiting system. The main challenge lies in the fact that the limiting system is degenerate~\cite{dibenedetto2012degenerate}: it takes the form
$
\partial_t \mathcal{U} + \mathcal{A} \mathcal{U}=0,
$
where the differential operator $\mathcal{A}$ is of second order with respect to the orientational variable and first order with respect to the spatial variable. This mixed structure complicates the application of standard powerful tools, such as semigroup theory.
To address this issue, we approximate the operator $\mathcal{A}$ with a non-degenerate version and derive the necessary \textit{a priori} bounds to establish the well-posedness of the limiting problem.

Finally, we consider an alternative method for deriving the limiting system based on an orthogonal decomposition ansatz for the pre-limiting solution. The key idea is to represent the solution as the sum of two orthogonal components: a boundary-layer part and a bulk distribution part, with the latter being orthogonal to boundary-layer-like functions.
As a result, we are able to prove weak $L^2$-convergence (noting that the spatial domain is unbounded) of this decomposition toward a solution of the form \eqref{lim-representation}
in the case $D_{\mathrm{rot}} = 0$.

In this work, we restrict our analysis to the case in which active rods are unable to escape from the wall. Mathematically, this corresponds to the condition $\boldsymbol{V} \cdot \boldsymbol{n} \geq g > 0$ for some constant $g > 0$. For instance, this condition is satisfied when an external force field $\boldsymbol{F}_{\text{ext}}$, directed toward the wall and exceeding the propulsion force, is applied, while the background flow $\boldsymbol{u}_{\text{BG}}$ in~\eqref{eqn:an-example} vanishes at the wall due to the no-slip boundary condition.
Relaxing this constraint in the proof of~\eqref{lim-representation} is challenging, as allowing active rods to escape introduces new boundary layers associated with escape orientations. 

Rigorous studies of singular limits for kinetic equations in bounded domains were discussed for the hydrodynamic limit from the Boltzmann equation to fluid models in Refs.~\cite{saint2009hydrodynamic,golse2004navier}. As pointed out in Ref.~\cite{berlyand2020kinetic}, in the system~\eqref{eqn:rescaled_main_fp}-\eqref{eqn:bc}, where the active rod has the ability to escape, boundary layers arise on a nontrivial scale distinct from~$\varepsilon^1$, analogous to the corner and parabolic layers characteristic of small-diffusion problems involving both advection and boundaries~\cite{shih1987asymptotic,van2006elliptic,gie2013analysis,gie2018singular}.
  The other important distinguishing feature of the system~\eqref{eqn:rescaled_main_fp}-\eqref{eqn:bc} is the anisotropic Laplacian leading to degenerate problem in the limit $\varepsilon\to 0$.
  The development of efficient numerical schemes for anisotropic Laplacians have been studied, e.g.,  in Refs.~\cite{degond2012asymptotic,degond2012duality,tang2017asymptotic} and the rigorous asymptotic analysis (without advection) in Ref.~\cite{ling2018asymptotic}. However, to the best of our knowledge, such singular limits in the presence of advection and non-penetrable boundaries, resulting in a coupled bulk-boundary limit system of the type obtained here, have not been addressed before. 
The limiting problem and main results are formulated in Section~\ref{sec:main-results}. Section~\ref{sec:well-posed} contains the proof of well-posedness of the limiting system.  Section~\ref{sec:l1} derives the limiting system and rigorously proves the $L^1$-convergence of the composite expansion to the solution of the limiting problem. Section~\ref{sec:l2} proves weak $L^2$-convergence in the particular case with no rotational diffusion. 

\section{Main results} \label{sec:main-results}
We consider the Fokker-Planck equation in half-space $\R^2_+=\left\{(x,y):y>0\right\}$. We impose the following simplifying assumptions on the advection $\boldsymbol{V}$ and reorientation $\Phi$ terms as follows: 
\begin{itemize}
\item[{({\bf  A1})}] $\boldsymbol{V}$ depends on orientation angle $\varphi$ only: $\boldsymbol{V}=V(\varphi)\boldsymbol{e}_y$, where $V\in C^1[0,2\pi]$  is a periodic function and it satisfies
          \begin{equation}\label{assumption-on-V}
              V(\varphi)\ge g>0, \quad\varphi\in [0,2\pi)
          \end{equation}
         for some positive constant $g>0$.  
\item[{({\bf A2})}] $\Phi=\Phi(\varphi)$ and $\Phi\in C^1_{\text{per}}[0,2\pi]$ is periodic.
\end{itemize}
For clarity of presentation, we assume that $\boldsymbol{V}$ and $\Phi$ are independent of the spatial variables $x$ and $y$. We note that our results can be extended to the case where $\boldsymbol{V}$ and $\Phi$ vary smoothly in space. 
Note also that assumptions (A1) and (A2) are satisfied, for example, in the case of~\eqref{eqn:an-example} with the shear flow $\boldsymbol{u}_{\text{BG}} = \gamma y\, \boldsymbol{e}_x$ and the wall-attracting external force $\boldsymbol{F}_{\mathrm{ext}} = -g\, \boldsymbol{e}_y$, where $\gamma > 0$ and $g > 1$ are given constants. In this case, we have
$
V(\varphi) = V_{\text{prop}} \sin(\varphi) - g$ and $\Phi(\varphi) = -\gamma \sin^2(\varphi),
$
which satisfy conditions (A1) and (A2). This case is also relevant when the problem in $\mathbb{R}^2_+$ is viewed as an approximation of the vicinity of a point on the wall, where the linear size of the region is much larger than the length of an active rod. In such a setting, it is reasonable to treat the wall as flat and to approximate the background Poiseuille flow by its linearization, which corresponds to a shear flow.

 Under this assumption, we can exclude dependence on variable $x$ from the unknown probability distribution function $f_\varepsilon$. Then the problem (\ref{eqn:rescaled_main_fp},\ref{eqn:bc}) reduces to~
\begin{equation}
\left\{
\begin{array}{l}\partial_t f_\varepsilon + \partial_\varphi(\Phi f_\varepsilon) - D\partial_\varphi^2 f_\varepsilon -V\partial_y f_\varepsilon-\varepsilon \partial_y^2 f_\varepsilon = 0, \\
-Vf_\varepsilon-\varepsilon \partial_y f|_{y=0} = 0.
\end{array}
\right.\label{eq:main_fp}
\end{equation}

It was shown in Ref.~\cite{berlyand2020kinetic} by means of formal multi-scale asymptotic expansion that the solution $f_\varepsilon$ of system \eqref{eq:main_fp} converges to $\rho_{\text{bulk}}+\delta_{\partial \mathbb R^2_+}\rho_{\text{wall}}$ in the limit $\varepsilon \to 0$, and functions $\rho_{\text{bulk}}(t,y,\varphi)$ and $\rho_{\text{wall}}(t,\varphi)$ solve the following equations:  
\begin{equation}
\left\{
\begin{array}{l}
\partial_t \rho_{\text{bulk}}+\partial_\varphi\left(\Phi \rho_{\text{bulk}}\right)-D\partial_\varphi^2 \rho_{\text{bulk}} -V\partial_y \rho_{\text{bulk}} = 0, \\
\partial_t \rho_{\text{wall}}+\partial_\varphi\left(\Phi \rho_{\text{wall}}\right)-D\partial_\varphi^2 \rho_{\text{wall}} =V\rho_{\text{bulk}}|_{y=0}. 
\end{array}
\right.
\label{eq:main_degenerate}
\end{equation}
The system is subject to initial conditions: 
\begin{equation*}
\rho_{\text{bulk}}|_{t=0}=\rho_{\text{bulk}}^{(0)} \text{ and }
\rho_{\text{wall}}|_{t=0}=\rho_{\text{wall}}^{(0)}.
\end{equation*}

Unlike the original initial-boundary value problem~\eqref{eq:main_fp}, which is a parabolic equation with no-flux boundary conditions, the system~\eqref{eq:main_degenerate} is non-classical in several important respects. 
First, it is \emph{degenerate} in the sense that it lacks a second derivative in the variable $y$. In other words, if we treat $y$ and $\varphi$ as spatial variables, the equation resembles a parabolic one, but with an incomplete Laplacian operator.
Second, rather than specifying boundary conditions for the bulk distribution $\rho_{\text{bulk}}$, the model incorporates a separate Fokker-Planck equation defined on the boundary of the domain.
In Section~\ref{sec:well_posedness}, we establish the following well-posedness result for the limiting system~\eqref{eq:main_degenerate}.

\begin{theorem}[Well-posedness \& regularity of system \eqref{eq:main_degenerate}]
Let assumptions (A1) \& (A2) hold. Assume also that $\rho_{\mathrm{bulk}}^{(0)}\in H^2_\varphi H^1_y$ and $\rho_{\mathrm{wall}}^{(0)}\in H^1_\varphi$. Then there exists a unique solution of the system \eqref{eq:main_degenerate} such that $\rho_{\mathrm{bulk}}\in C([0,T];H^2_\varphi H^1_y)$ and $\rho_{\mathrm{wall}}\in C([0,T];H^2_\varphi)$. Here, $H^2_\varphi=H^2_{\text{per}}(0,2\pi)$ and $ H^2_\varphi H^1_y$ is the space of functions $v(y,\varphi)$, $2\pi$-periodic in $\varphi$, such that 
\begin{equation*}
     \int\limits_0^{2\pi} \int\limits_0^{\infty} |D^2_\varphi  v|^2 + |D_\varphi v|^2 + |v|^2 + |D_y v|^2 \,\mathrm{d}y\mathrm{d}\varphi < \infty.
\end{equation*}
Moreover, for every integer $k\geq 0$, if $\rho_{\mathrm{bulk}}^{(0)}\in H^{k+2}_\varphi H^{k+1}_y$ and $\rho_{\mathrm{wall}}^{(0)}\in H^{k+1}_\varphi$, then $\rho_{\mathrm{bulk}}\in C([0,T];H^{k+2}_\varphi H^{k+1}_y)$ and $\rho_{\mathrm{wall}}\in C([0,T];H^{k+2}_\varphi)$.

\label{thm:degenerate}
\end{theorem}

Our main focus is justifying the limit of initial-boundary value problem \eqref{eq:main_fp} to system \eqref{eq:main_degenerate}. 

\begin{theorem}[$L^1$-approximation] 
Let assumptions (A1) \& (A2) hold. Let $f_\varepsilon$ and $(\rho_\mathrm{bulk},\rho_{\mathrm{wall}})$ be solutions of \eqref{eq:main_fp} and \eqref{eq:main_degenerate}, respectively. Assume also that $f_\varepsilon|_{t=0}=\rho_{\mathrm{bulk}}|_{t=0}$ is independent of $\varepsilon$ and $\rho_{\mathrm{wall}}|_{t=0} = 0$. Then define an approximation $\bar{f}_\varepsilon$ of $f_\varepsilon$ as 
\begin{equation}
\bar{f}_\varepsilon(t,y,\varphi):=\rho_{\mathrm{wall}}(t,\varphi)\cdot\dfrac{V(\varphi)}{\varepsilon} e^{-V(\varphi)y/\varepsilon}+\rho_{\mathrm{bulk}}(t,y,\varphi).\label{def-of-l1-approx}
\end{equation}
Then for all $t>0$ 
\begin{equation}
\lim\limits_{\varepsilon \to 0}\int\limits_{0}^{2\pi}\int\limits_{0}^{\infty}|f_\varepsilon(t,y,\varphi) - \bar{f}_\varepsilon(t,y,\varphi)|\,\mathrm{d}y\mathrm{d}\varphi = 0.\label{statement-of-l1-approx}
\end{equation}
\label{thm:l1convergence}
\end{theorem}

\begin{remark}
Note that the term $\dfrac{V(\varphi)}{\varepsilon} e^{-V(\varphi)y/\varepsilon}$ on the right-hand side of \eqref{def-of-l1-approx} serves as an approximation of the Dirac delta function. Taking this into account, one can observe the correspondence between the definition \eqref{def-of-l1-approx} and the limiting representation given in \eqref{lim-representation}.
\end{remark}

\begin{remark}
The assumption that $f_\varepsilon|_{t=0} = \rho_{\mathrm{bulk}}|_{t=0}$ is independent of $\varepsilon$, and that $\rho_{\mathrm{wall}}|_{t=0} = 0$, reflects a natural physical situation at the initial moment of time: the rods are distributed in the bulk and have not yet accumulated at the wall. However, this assumption can be significantly relaxed. In general, our proof shows that the equality in \eqref{statement-of-l1-approx} holds for all $t > 0$ provided it holds at $t = 0$.  
\end{remark}

Next, we turn to the case of the zero rotational diffusion $D=0$: 
\begin{equation}
\left\{
\begin{array}{l}\partial_t f_\varepsilon + \partial_\varphi(\Phi f_\varepsilon)  -V\partial_y f_\varepsilon-\varepsilon \partial_y^2 f_\varepsilon = 0, \\
-Vf_\varepsilon-\varepsilon \partial_y f|_{y=0} = 0.
\end{array}
\right.\label{eq:main_fp-noD}
\end{equation}
In this case, using a different approach from the one used in the proof of Theorem~\ref{thm:l1convergence} and a key ansatz for the solution $f_\varepsilon$, we proved weak $L^2$-convergence of $f_\varepsilon$ to a singular representation \eqref{lim-representation}. 

\begin{theorem}[weak $L^2$-convergence]
Let assumptions (A1) \& (A2) hold. Let $f_\varepsilon$ be a solution of \eqref{eq:main_fp-noD}. Introduce $m_\varepsilon$ and $u_\varepsilon$ such that 
\begin{equation}
    \label{eq:forthogonal}
     f_\varepsilon = \dfrac{1 }{\varepsilon} m_\varepsilon(t,\varphi) e^{-Vy/\varepsilon} + u_\varepsilon(t,\varphi, y) \quad \text{with} \quad \int\limits_0^\infty u_\varepsilon e^{-Vy/\varepsilon} \, \mathrm{d}y = 0.
\end{equation}
Then
        \begin{align*}
            m_\varepsilon &\rightharpoonup m \text{ in $L^2(0,2\pi)$}\\
            u_\varepsilon &\rightharpoonup u \text{ in $H^1(0,\infty) \times L^2(0,2\pi)$.}
        \end{align*}
        and $\rho_{\mathrm{wall}}:=m/V$ and $\rho_{\mathrm{bulk}}:=u$ satisfy the limiting equations:
        \begin{equation}
\begin{cases}
    \partial_t \rho_{\mathrm{wall}}  + \partial_\varphi(\Phi \rho_{\mathrm{wall}}) - V\partial_y \rho_{\mathrm{wall}} = 0,\\
    \partial_t\rho_{\mathrm{bulk}} + \partial_\varphi \left(\Phi \rho_{\mathrm{bulk}}\right) = V\rho_{\mathrm{wall}} \vert_{y=0}.
\end{cases}
    \label{eq:limiting-noD}
\end{equation}
\label{thm:weakL2}
\end{theorem}

\begin{remark}
Given $f_\varepsilon$, functions $m_\varepsilon$ and $u_\varepsilon$ are uniquely defined by \eqref{eq:forthogonal}: 
\begin{equation}\label{eq:representation}
m_\varepsilon =2V \int\limits_0^\infty f_\varepsilon e^{-V(\varphi) y/\varepsilon}\,\text{d}y \text{ and }u_\varepsilon = f_\varepsilon - \dfrac{1  }{\varepsilon} m_\varepsilon(t,\varphi) e^{-Vy/\varepsilon}.
\end{equation}
\end{remark}

\begin{remark}
As is evident from the proof of Theorem~\ref{thm:l1convergence}, the term $ (1/\varepsilon) m_\varepsilon e^{-V y / \varepsilon} $ characterizes the boundary layer structure. It satisfies the pre-limiting equation at order $ \varepsilon^{-1} $, namely
$$
- V \partial_y f_\varepsilon - \varepsilon \partial_y^2 f_\varepsilon = 0,
$$
as well as the no-flux boundary condition. The orthogonality condition on $ u_\varepsilon $ in \eqref{eq:forthogonal} implies that $ u_\varepsilon $ does not contain a boundary layer component (in the sense of the $ L^2 $ inner product) and instead represents the bulk distribution. For $ D > 0 $, the lack of \textit{a priori} $ L^2 $ estimates for $ m_\varepsilon $ and $ u_\varepsilon $ suggests that a more refined formulation is required.

\end{remark}





\section{Proof of Theorem~\ref{thm:degenerate}: Well-posedness of the limiting problem}\label{sec:well-posed}
\label{sec:well_posedness}
\subsection{Formulation of the limiting problem}

In this section, we study the limiting problem for the bulk and wall  distributions, $u(t,y,\varphi)=\rho_{\text{bulk}}$ and  $w(t,\varphi)=\rho_{\text{wall}}$, respectively. The limiting problem is 
\begin{equation}
\left\{
\begin{array}{l}
\partial_t u + \partial_\varphi(\Phi u) - D \partial_\varphi^2 u -  V\partial_y u = 0, \quad y>0, \, t>0, \, 0 \leq \varphi < 2\pi,\\
\partial_t w + \partial_\varphi(\Phi w) - D \partial_\varphi^2 u  = Vu|_{y=0},\quad y=0, \, t>0, \, 0 \leq \varphi < 2\pi.\\
\end{array}
\right.\label{eqn-1Dlimiting}
\end{equation}
We aim to show the well-posedness result for system \eqref{eqn-1Dlimiting} using the semi-group theory. 
We begin by fixing notations

\begin{itemize}
    \item $I:= [0,2\pi)$ and $L^2_\varphi:=L^2(I)$,
    \item $L^2_y:=L^2(\mathbb R)$.
    \item We use the shorthand $L^2_\varphi L^2_y$ to denote the space of functions $g(\varphi,y)$ with finite
\begin{equation}
\|f\|_{L^2_\varphi L^2_y} = \left(\int\limits_0^{2\pi}\int\limits_\mathbb{R} |f|^2 \, \text{d}y \, \text{d}\varphi\right)^{1/2}
\end{equation}
We will also denote $X:=L^2_\varphi L^2_y$. 

\noindent Higher order Sobolev spaces are introduced similarly with the account of periodic boundary conditions with respect to variable $\varphi\in [0,2\pi)$ and zero boundary conditions for $|y|\to \infty$. For example, 
\begin{align}
    H^2_\varphi H^1_y: = \bigg\{v(y,\varphi)& \bigg\vert \int\limits_0^{2\pi} \int\limits_\mathbb{R} |D^2_\varphi  v|^2 + |D_\varphi v|^2 + |v|^2 + |D_y v|^2 \text{d}y\text{d}\varphi < \infty,\nonumber \\
    &v(0,y)=v(2\pi,y) \text{ for } y\in\mathbb R,
    \quad
    \lim_{|y|\to\infty}v(\varphi,y)=0 \text{ for } \varphi\in I\bigg\}
\end{align}

\end{itemize}


We rewrite \eqref{eqn-1Dlimiting} in the form 
\begin{equation}
\partial_t \mathcal{U}+\mathcal{L} \mathcal{U} = 0, 
\end{equation}
where $\mathcal{U} = \left\{u,v\right\}$ and the operator $\mathcal{L}$ is defined as follows: 
\begin{equation}
\mathcal{L}\mathcal{U} = 
\left[
\begin{array}{cc}
B & 0 \\
VT & A
\end{array}
\right]\left(\begin{array}{c}u\\w\end{array}\right)
\end{equation}
Here, operators $A$, $B$, and $T$ are 
\begin{equation}\label{eq:operators}
    \begin{aligned}
        Aw &=\partial_\varphi(\Phi w)-D\partial_\varphi^2 w,
        &\qquad A:&\, D(A)=H_\varphi^2\subset L^2_\varphi\to L^2_\varphi,\\
        Bu &=\partial_\varphi(\Phi u)-D\partial_\varphi^2 u+V\partial_y u,
        &\qquad B:&\,D(B)=H_\varphi^2 H^1_y\subset L^2_\varphi L^2_y\to L^2_\varphi L^2_y,\\
        Tu &=u|_{y=0}& \qquad T:&\,D(T)= H_\varphi^2 H^1_y\subset L^2_\varphi L^2_y \to L^2_\varphi. 
    \end{aligned}
\end{equation}
Now, we can define the domain of operator $\mathcal{L}:\, D(\mathcal{L
})\subset \mathcal{X}\to \mathcal{X}$ and the space $\mathcal{X}$ this operator acts on: 
\begin{equation}
D(\mathcal{L}) = D(B)\otimes D(A) \text{ and }\mathcal{X}= X \otimes L^2_\varphi= L^2_\varphi L^2_y \otimes L^2_\varphi. 
\end{equation}



\subsection{$B$ is an infinitesimal generator of $C_0$-semigroup} 
\label{sec:b_c0_semigroup}
Since the first equation in \eqref{eqn-1Dlimiting} is independent of $v$ and is of the form $\partial_t u + Bu =0$, we consider its well-posedness separately. To this end, we aim to show in this subsection that $B$ is an infinitesimal generator of $c_0$-semigroup. According to Lumer-Phillips Theorem \cite{pazy2012semigroups} it is sufficient to show that there exists $\mu>0$ such that
\begin{itemize}

\item[({\it i})] ({\it Accretivity}) $\langle [B+\mu]u,u\rangle_{X}\geq 0$ for all $u\in D(B)$;

\item[({\it ii})] ({\it Maximality}) $\mathrm{Range}(B+\mu)=L^2_yL^2_\varphi$, that is, for any $f\in L^2_yL^2_\varphi$ there exists $u\in D(B)$ such that $(B+\mu)u=f$. 

\end{itemize}


\begin{lemma}[Accretivity of $B$] \label{lem:coercive-bound}
There exists $\mu_0>0$ such that for every $\lambda>\mu_0$ and every $u\in D(B)$,
\begin{equation}\label{eq:coercive-bound}
\langle(B+\lambda)u,\,u\bigr\rangle_X \ge (\lambda-\mu_0)\|u\|_X^2
+ D\|\partial_\varphi u\|_X^2.
\end{equation}
\end{lemma}

\begin{proof}
Let $u\in D(B)$.  
Integration by parts in $\varphi$ and in $y$, together with periodic boundary conditions and decay at $|y|\to\infty$, gives
\begin{eqnarray*}
\langle B\,u,\,u\rangle_X
&=&\int\limits_{0}^{2\pi}\int\limits_{\mathbb R}\partial_\varphi (\Phi u) \, u\,\text{d}y\text{d}\varphi-D\int\limits_{0}^{2\pi}\int\limits_{\mathbb R}\partial_\varphi^2 u \, u\,\text{d}y\text{d}\varphi+\int\limits_{0}^{2\pi}V\int\limits_{\mathbb R}\partial_y u \, u\,\text{d}y\text{d}\varphi
\\
&=&\dfrac{1}{2}\int\limits_0^{2\pi}\!\!\int\limits_{\mathbb{R}}\Phi'\,u^2
\,\mathrm{d}y\mathrm{d}\varphi
+ D\int\limits_0^{2\pi}\!\!\int\limits_{\mathbb{R}}|\partial_\varphi u|^2
\,\mathrm{d}y\mathrm{d}\varphi\\
&\geq & -\dfrac{1}{2}\|\Phi'\|_{L^{\infty}}\|u\|_X^2+D\|\partial_\varphi u\|_{X}^2. 
\end{eqnarray*}
Then, denoting $\mu_0: = \dfrac{1}{2}\|\Phi'\|_{L^\infty} + 1$ we obtain
\begin{equation*}
\langle(B+\lambda )u,\,u\rangle_X
= \lambda\|u\|_X^2
+ \langle B\,u,\,u\rangle_X\geq (\lambda-\mu_0)\|u\|_X^2
+ D\|\partial_\varphi u\|_X^2.
\end{equation*} 

\end{proof}

\noindent Lemma~\ref{lem:coercive-bound} implies accretivity of $B$. The proof of maximality is split into three steps. First, we prove existence of a solution in $H^1_\varphi H^1_y$ for the regularized operator $B_\varepsilon:=B-\varepsilon \partial_y^2$.  Second, we obtain bounds uniform in $\varepsilon$. Third, we conclude via a standard compactness argument as $\varepsilon \to 0$ the limiting solution satisfies $ (B +\mu)u = f$, for all $f\in X$.

\begin{lemma}[Maximality of $B$]\label{lem:range-explicit}
Let $\lambda>\mu_0$.  For each $\varepsilon>0$ define
\begin{equation*}
B_\varepsilon := B - \varepsilon\,\partial_y^2,
\end{equation*}
and set $V = H^1_\varphi H^1_y$.

\smallskip 

\noindent Then for every $f\in X$ the regularized equation
\begin{equation}\label{eq:shifted_regularized_B}
( B_\varepsilon+\lambda)\,u_\varepsilon = f
\end{equation}
has a unique solution $u_\varepsilon\in V$, and there is $C$, independent of $\varepsilon$, so that
\begin{equation}
\|u_\varepsilon\|_X^2 
+ \|\partial_\varphi u_\varepsilon\|_X^2 
+ \|\partial_y u_\varepsilon\|_X^2
+\|\partial_\varphi^2 u_\varepsilon\|^2_X \le C\,\|f\|_X^2.\label{main_ineq_for_B_eps}
\end{equation}
Moreover, there exists a subsequence of $u_\varepsilon$ converging weakly in $H^2_\varphi H^1_y$ and strongly in $X$ to $u\in D(B)$ solving
\begin{equation}
(B+\lambda)\,u = f.
\end{equation}
\end{lemma}

\begin{proof}
$~$\\
\underline{{STEP 1.}} {\it Existence and uniqueness for the regularized problem.}
Define the bilinear form for $u,v\in V$:
\begin{equation*}
a_\varepsilon(u,v)
:= \lambda\langle u,v\rangle_X
+ b(u,v)  
+ \varepsilon\,\langle \partial_y u,\partial_y v\rangle_X,
\end{equation*}
where $b(u,v)$ is the bilinear form corresponding to the operator $B$, namely, 
\begin{equation*}
b(u,v):=\langle \partial_\varphi\left(\Phi u\right),v\rangle_X +D \langle \partial_\varphi u, \partial_\varphi v \rangle_X +\langle V\partial_y u, v \rangle_X.
\end{equation*}
By the same arguments as in the proof of Lemma \ref{lem:coercive-bound} we get
\begin{equation*}
a_\varepsilon(u,u)
\ge (\lambda-\mu_0)\|u\|_X^2
+ D\,\|\partial_\varphi u\|_X^2
+ \varepsilon\,\|\partial_y u\|_X^2\geq \min\{\lambda -\mu_0,D,\varepsilon\}\|u\|_V^2
\end{equation*}
implying coercivity of the form $a_\varepsilon$. Continuity of the form $a_\varepsilon$, $a_\varepsilon(u,v)\leq C\|u\|_V\|v\|_V$ is straightforward. Then the Lax–Milgram theorem yields a unique $u_\varepsilon\in V$ solving
\begin{equation} \label{eqn:regularized-weak-form}
a_\varepsilon(u_\varepsilon,v) = \langle f,v\rangle_X
\quad\forall v\in V.
\end{equation}

\smallskip 

\noindent\underline{STEP 2.}  \textit{A priori estimates for $u_\varepsilon$, independent of $\varepsilon$.} $~$ \\ {\bf First}, substitute $v = u_\varepsilon$ into \eqref{eqn:regularized-weak-form} and use coercivity of $a_\varepsilon$ established in the previous step:
\begin{equation*}
(\lambda-\mu_0)\|u_\varepsilon\|_X^2
+ D\,\|\partial_\varphi u_\varepsilon\|_X^2
+ \varepsilon\,\|\partial_y u_\varepsilon\|_X^2\leq \langle f,u_\varepsilon\rangle_X.
\end{equation*}

\noindent Using Young's inequality to estimate the right-hand side: 
\[
\langle f,u_\varepsilon\rangle_X
\le \dfrac{\lambda-\mu_0}{2}\,\|u_\varepsilon\|_X^2
+ \dfrac{1}{2(\lambda-\mu_0)}\,\|f\|_X^2,
\]
we obtain
\begin{equation}\nonumber
\dfrac{\lambda-\mu_0}{2}\,\|u_\varepsilon\|_X^2
+ D\|\partial_\varphi u_\varepsilon\|_X^2
+ \varepsilon\|\partial_y u_\varepsilon\|_X^2
\le \dfrac{1}{2(\lambda-\mu_0)}\,\|f\|_X^2.
\end{equation}
This inequality implies an {\it a priori} bound
\begin{equation}
\label{eq:C-u_dphiu_estimate}
\|u_\varepsilon\|_X^2
+ \|\partial_\varphi u_\varepsilon\|_X^2
\le C\,\|f\|_X^2
\end{equation}
with a positive constant $C$ depending on $\lambda -\mu_0$ and $D$, but independent of $\varepsilon$. 

\smallskip

\noindent {\bf Next}, substitute $v = \partial_yu_\varepsilon$ into \eqref{eqn:regularized-weak-form}. Note that this substitution is admissible since $f\in X = L^2_\varphi L^2_y$, then, similar to Section \S 6.3.1 in Ref.~\cite[\S 6.3.1]{evans2022partial}, $u_\varepsilon \in H^2_\varphi H^2_y$ and substitution into the weak formulation \eqref{eqn:regularized-weak-form} can be treated as multiplication of the strong formulation $(\lambda + B_\varepsilon)u_\varepsilon = f$ by $v = \partial_yu_\varepsilon$. We obtain
\begin{equation}\label{eqn:3-15}
\lambda\,\langle u_\varepsilon, \partial_y u_\varepsilon\rangle_X
+ \langle B\,u_\varepsilon,\partial_yu_\varepsilon\rangle_X+ \varepsilon\,\langle \partial_y^2 u_\varepsilon, \partial_y u_\varepsilon\rangle_X
= \langle f,\partial_yu_\varepsilon\rangle_X.
\end{equation}
Next, due to assumption \eqref{assumption-on-V} we have
\begin{eqnarray*}
\langle B\,u_\varepsilon,\partial_yu_\varepsilon\rangle_X &=& \langle \partial_\varphi \left(\Phi u_\varepsilon\right),\partial_y u_\varepsilon\rangle_X + D\langle \partial_\varphi^2 u_\varepsilon, \partial_y u_\varepsilon\rangle_X+\langle V\partial_y u_\varepsilon, \partial_y u_\varepsilon \rangle_X \\
&&\geq -\dfrac{\|\Phi\|_{L^{\infty}}^2}{g}\|\partial_\varphi u_\varepsilon\|_{X}^2 + \dfrac{3g}{4}\|\partial_yu_\varepsilon\|_X^2.
\end{eqnarray*}
Using this bound in \eqref{eqn:3-15} and noting that due to the derivative of square formula ($\partial_y (v^2)=2v\partial_yv$) and the fundamental theorem of calculus, the first and third terms in the equality above vanish, we obtain
\begin{equation}
\|\partial_y u_\varepsilon\|_X^2 \leq \dfrac{16}{9g^2}\|f\|^2_X+\dfrac{8\|\Phi\|_{L^{\infty}}}
{3g^2}\|\partial_\varphi u_\varepsilon\|_{X}^2.
\end{equation}
Using \eqref{eq:C-u_dphiu_estimate} we get the existence of a positive constant $C$ (may be different from the one in \eqref{eq:C-u_dphiu_estimate}), depending on $\lambda-\mu_0$, $D$, $g$, and $\|\Phi\|_{L^{\infty}}$ but independent of $\varepsilon$, such that 
\begin{equation}
\label{eq:C-u_dy_estimate}
\|u_\varepsilon\|_X^2
+ \|\partial_\varphi u_\varepsilon\|_X^2+\|\partial_y u_\varepsilon\|_X^2
\le C\,\|f\|_X^2.
\end{equation}

\smallskip 

\noindent {Finally,} we obtain the bound for $\partial_\varphi^2 u_\varepsilon$ from \eqref{main_ineq_for_B_eps}. To this end, rewrite the equation \eqref{eq:shifted_regularized_B} for $u_\varepsilon$ in the form
\begin{equation}\label{eq:elliptic-phi-ode-phi2}
- D\partial_\varphi^2u_\varepsilon - \varepsilon \partial_y^2 u_\varepsilon= g,
\end{equation}
where
$
g
:= f
  - \lambda\,u_\varepsilon
  - \partial_\varphi\bigl(\Phi\,u_\varepsilon\bigr)
  + V\,\partial_y u_\varepsilon.
$
Squaring the left-hand side of \eqref{eq:elliptic-phi-ode-phi2} and integrating by parts, we get
\begin{eqnarray*}
\|- D \partial_\varphi^2 u_\varepsilon -\varepsilon\partial_y^2 u_\varepsilon\|_X^2 &=& D^2 \|\partial_\varphi^2 u_\varepsilon\|^2_X + \varepsilon^2 \| \partial_y^2 u_\varepsilon\|^2_X +2D\varepsilon \int\limits_0^{2\pi}\int\limits_{\mathbb R} \partial_\varphi^2 u_\varepsilon \partial_y^2 u_\varepsilon \,\text{d}y\text{d}\varphi \\
 &=&D^2 \|\partial_\varphi^2 u_\varepsilon\|^2_X +2D\varepsilon \int\limits_0^{2\pi}\int\limits_{\mathbb R} (\partial_\varphi\partial_y u_\varepsilon)^2\,\text{d}y\text{d}\varphi \\ &\geq &  D^2 \|\partial_\varphi^2 u_\varepsilon\|^2_X.
\end{eqnarray*}
On the other hand, the right-hand side of \eqref{eq:elliptic-phi-ode-phi2} can be estimated by an expression of the form $C\|f\|^2_{X}$ in view of the bound \eqref{eq:C-u_dy_estimate}. Hence, there exists a constant $C$, independent of $\varepsilon$, such that $\|\partial_\varphi^2 u_\varepsilon\|^2_X\leq C\|f\|^2_X$. Thus, we proved  the estimate \eqref{main_ineq_for_B_eps} from formulation of this Lemma.

\smallskip 

\noindent\underline{STEP 3.}  \textit{Passing to the limit in \eqref{eq:shifted_regularized_B} with respect to $\varepsilon\to 0$.} $~$ \\
The estimate \eqref{main_ineq_for_B_eps} implies a weakly converging subsequence of $\{u_\varepsilon\}$ to $u\in D(B)$ solving $(B+\lambda)\,u = f$.
 \end{proof}

\subsection{Solvability of the limiting system \eqref{eqn-1Dlimiting}}
\noindent We start with solvability of the first equation in the limiting system \eqref{eqn-1Dlimiting} which is a direct consequence of that we established that the operator $B$ is the infinitesimal generator of a $c_0$-semigroup \cite{pazy2012semigroups,evans2022partial}.  
\begin{lemma}[Solvability of   $\partial_tu+Bu=0$]\label{lem:range-explicit} For any $u_0\in H^2_\varphi H^1_y$ and $T>0$, there exists a unique solution of $u\in C([0,T];H^2_\varphi H^1_y)$ satisfying the following initial-value problem:
\begin{equation}\label{eqn-for-u}
\left\{
\begin{array}{l}
u_t + \partial_\varphi \bigl(\Phi u\bigr)-D\partial_\varphi^2 u -V \partial_y u  = 0, \\
u|_{t=0} = u_0.
\end{array}
\right.
\end{equation}

\end{lemma}

\noindent Next, we turn to the second equation in \eqref{eqn-1Dlimiting} whose solution can be written in the form of the Duhamel's formula: 
\begin{equation}
w(t) = e^{-At}w_0 + \int\limits_0^{t}e^{-A(t-\tau)}T(Vu)(\tau)\,\text{d}\tau.   
\end{equation}
Here, the dependence on variable $\varphi$ is omitted for clarity. The evolution operator $e^{-At}:L^2_\varphi\to L^2_\varphi$ is well-defined since $A$ is the infinitesimal generator of a $c_0$-semigroup. Indeed, $A$ is an elliptic operator with periodic boundary conditions whose accretivity can be shown similar to Lemma~\ref{lem:coercive-bound} and maximality follows directly from the application of the Lax-Milgram theorem. The trace operator $T$ is defined in the third line of  \eqref{eq:operators}. By the trace theorem, we have $T(Vu)\in C([0,T];L^2_\varphi)$ and thus, due to the classical theory of non-homogeneous equations of linear abstract equations whose linear part the infinitesimal generator of a $c_0$-semigroup\cite[Section 4]{pazy2012semigroups}, we get the following well-posedness result for the equation for $u$.

\begin{lemma}[Solvability of   $\partial_tw+Aw=T(Vu)$]\label{lem:range-explicit} Let $u\in C([0,T];H^2_\varphi H^1_y)$ be a solution of \eqref{eqn-for-u}. For any $w_0\in H^2_\varphi$ and $T>0$, there exists a unique solution of $w\in C([0,T];H^2_\varphi )$ satisfying the following initial-value problem:
\begin{equation}
\left\{
\begin{array}{l}
w_t + \partial_\varphi \bigl(\Phi w\bigr)-D\partial_\varphi^2 w =V u|_{y=0}, \\
w|_{t=0} = w_0.
\end{array}
\right.
\end{equation}

\end{lemma}

\smallskip 

\noindent The regularity result for \( k = 1 \) is established by formally differentiating the first equation in \eqref{eq:main_degenerate} with respect to \( y \) and \( \varphi \), yielding equations for \( \partial_y u \) and \( \partial_\varphi u \) of the form
$
\partial_t v + Bv = F(u),
$
where \( v = \partial_y u \) or \( v = \partial_\varphi u \), and \( F(u) \in C([0,T]; X) \). The result then follows by applying the classical theory of non-homogeneous evolution equations~\cite[Section~4]{pazy2012semigroups}. The case of general integer \( k > 1 \) is handled by induction.

\smallskip 

\noindent This concludes the proof of Theorem~\ref{thm:degenerate}.

\section{Proof of Theorem~\ref{thm:l1convergence}: $L^1$-convergence for $D> 0$} \label{sec:l1}




We first construct asymptotic expansions for the solution in two distinct regions: an outer region (or bulk) where the solution varies slowly, and an inner region (the boundary layer) where the solution varies rapidly. The overall dynamics are then found by enforcing a matching condition between these two expansions. Finally, we prove that the $L^1$ norm of the difference between the composite expansion and the original solution vanishes as $\varepsilon\to 0$.

\subsubsection*{Outer solution}

In the outer region, away from the boundary ($y > 0$), we assume a regular asymptotic expansion for the probability distribution function:
\begin{equation}
f_\varepsilon (t,\varphi, y) \sim  u_0(t,\varphi, y) + \varepsilon u_1(t,\varphi, y) + ...
\end{equation}
Substituting this expansion into the governing equation from \eqref{eq:main_fp} we get
\begin{equation*}
\partial_t f_\varepsilon + \partial_\varphi(\Phi f_\varepsilon) - D\partial_\varphi^2 f_\varepsilon -V\partial_y f_\varepsilon-\varepsilon \partial_y^2 f_\varepsilon = 0,
\end{equation*}
and collecting terms of like powers in $\varepsilon$ yields a hierarchy of equations: 
\begin{eqnarray*}
&& \partial_t u_0 + \partial_\varphi(\Phi u_0) - D\partial^2_\varphi u_0 - V \partial_y u_0 \\ 
&& \hspace{20pt} + \varepsilon \left[\partial_t u_1 + \partial_\varphi(\Phi u_1) - D\partial^2_\varphi u_1 - V \partial_y u_1 - \partial_y^2 u_0\right]\\
&& \hspace{20pt} + ... \\
&& \hspace{20pt} = 0.
\end{eqnarray*}
\noindent The leading-order, $O(1)$ \& $O(\varepsilon)$, problems govern the bulk distributions $u_0$ and $u_1$:
\begin{eqnarray}
&&\partial_t u_0 + \partial_\varphi(\Phi u_0) - D\partial^2_\varphi u_0 - V \partial_y u_0 = 0,\label{eqn-for-u_0}\\
&&\partial_t u_1 + \partial_\varphi(\Phi u_1) - D\partial^2_\varphi u_1 - V \partial_y u_1 = \partial_y^2 u_0.\label{eqn-for-u_1}
\end{eqnarray}

\subsubsection*{Inner solution $-$ Boundary layer at $y=0$}
We introduce a rescaled spatial variable $z = y/\eps$ and then the unknown function $f_\varepsilon$ at the boundary $\{y=0\}$ can be represented as $f_\varepsilon(t,\varphi,y)=p_\varepsilon(t,\varphi,z)$. In terms of the inner variable $z$, the governing equation with boundary conditions in \eqref{eq:main_fp} become:
\begin{equation}
\left\{
\begin{array}{l}
\partial_t p_\varepsilon + \partial_\varphi(\Phi p_\varepsilon) - D\partial^2_\varphi p_\varepsilon + \frac{1}{\eps}(-V\partial_z p_\varepsilon - \partial_z^2 p_\varepsilon) = 0, \\
\left[-Vp_\varepsilon-\partial_z p_\varepsilon\right]|_{z=0} = 0.
\end{array}
\right.\label{governing-pde-for-p-eps}
\end{equation} 
We expand $p_\varepsilon$ as follows starting from $\varepsilon^{-1}$ term: 
\begin{equation*}
p_\varepsilon (t, \varphi, z) \sim \eps^{-1} p_{-1}(t, \varphi, z) + p_0(t, \varphi, z) + ...
\end{equation*}
The leading $O(\eps^{-1})$ term is necessary to capture the steep gradient of $f_\varepsilon$ near the boundary. Substituting this expansions into the governing equation and the boundary condition in \eqref{governing-pde-for-p-eps} we get, respectively,  
\begin{eqnarray*}
&&\varepsilon^{-2}\left[-V \partial_z p_{-1} - \partial_z^2 p_{-1}\right] \\
&& \hspace{20pt}+\varepsilon^{-1}\left[\partial_t p_{-1} + \partial_\varphi ( \Phi p_{-1}) - D \partial_\varphi^2 p_{-1}-\left\{-V \partial_z p_{0} - \partial_z^2 p_{0}\right\}\right]\\
&& \hspace{20pt}+...\\
&& \hspace{25pt} = 0,
\end{eqnarray*}
and
\begin{eqnarray*}
\left[-V  p_{-1} - \partial_z p_{-1}\right]|_{z=0} + \varepsilon 
\left[-V p_{0} - \partial_z p_{0}\right]|_{z=0} + ...
\end{eqnarray*}

These expansions for $p_\varepsilon$ yield the following equations at leading  orders of $\varepsilon$. Specifically, we have the following boundary value problem for $p_{-1}$:
\begin{equation} \label{eq:p_neg1_bvp}
\begin{cases}
-V \partial_z p_{-1} - \partial_z^2 p_{-1} = 0, \quad & z > 0, \\
(-V p_{-1} - \partial_z p_{-1})|_{z=0} = 0.
\end{cases}
\end{equation}
and for $p_0$:
\begin{equation}\label{eq:p0_bvp}
\begin{cases}
V \partial_z p_0 + \partial_z^2 p_0 = \partial_t p_{-1} + \partial_\varphi ( \Phi p_{-1}) - D \partial_\varphi^2 p_{-1}, \quad z>0\\
(-V  p_{0} - \partial_z p_{0})|_{z=0} = 0
\end{cases}
\end{equation}

\subsubsection*{Matching asymptotic expansions}
The asymptotics are matched by their corresponding order of $\varepsilon$, noting that the inner and outer expansions must be consistent. This is enforced via a matching condition, which requires that the behavior of the inner solution far from the wall (in terms of variable $z$) corresponds to the short-range behavior of the outer solution:
\begin{equation}
\lim_{z \to \infty} p_k(t, \varphi, z) = u_k(t, \varphi, 0), \quad k = -1, 0, 1, \dots \label{matching-condition}
\end{equation}
We take $u_{-1} \equiv 0$.

Next, we note that the boundary value problem \eqref{eq:p_neg1_bvp} for $p_{-1}$ has the general solution: 
\begin{equation}
p_{-1}(t,\varphi,z) = \hat{p}_{-1}(t,\varphi)\cdot V(\varphi)e^{-zV(\varphi)}.
\label{formula-for-p--1}
\end{equation}
Here, we factor out $V(\varphi)$ to simplify calculations below. Note that $\int\limits_0^\infty Ve^{-zV}\,\text{d}z=1$ so $Ve^{-zV}$ can be viewed as an approximation of $\delta$-function and thus $\hat{p}_{-1}$ is the limiting wall distribution.

To determine the governing equation for $\hat{p}_{-1}$, we integrate equation \eqref{eq:p0_bvp} with respect to $z$ from $0$ to $\infty$:
\begin{equation}
\int\limits_0^\infty (V \partial_z p_0 + \partial_z^2 p_0) \,\text{d}z = \int\limits_0^\infty \left( \partial_t p_{-1} + \partial_\varphi ( \Phi p_{-1}) - D \partial_\varphi^2 p_{-1} \right) \,\text{d}z.\label{integral-eqn-for-p--1}
\end{equation}
The left-hand side of \eqref{integral-eqn-for-p--1} evaluates to:
\begin{equation*}
\int\limits_0^\infty (V \partial_z p_0 + \partial_z^2 p_0)\,\text{d}z =\left.\left[V p_0 + \partial_z p_0\right]\right|_0^\infty = V u_0|_{y=0}.
\end{equation*}

Here we used the matching condition for $p_0$ and the no-flux boundary condition from \eqref{eq:p0_bvp}. The right-hand side of \eqref{integral-eqn-for-p--1} becomes:
\begin{align*}
\int\limits_0^\infty \left(\partial_t + \mathcal{L}_\varphi \right) \left( \hat{p}_{-1} V e^{-Vz} \right)\,\text{d}z &= \left(\partial_t + \mathcal{L}_\varphi \right) \left[\hat{p}_{-1} \int\limits_0^\infty V e^{-Vz}\,\text{d}z\right] \\&
= \left(\partial_t + \mathcal{L}_\varphi \right) \hat{p}_{-1},
\end{align*}
where $\mathcal{L}_\varphi(\cdot) := \partial_\varphi(\Phi \, \cdot) - D\partial_\varphi^2(\, \cdot \, )$. Equating representations we found for the each side of \eqref{integral-eqn-for-p--1} we get the evolution equation for the wall distribution $\hat{p}_{-1}$:
\begin{equation}
\partial_t \hat{p}_{-1} + \partial_\varphi(\Phi \hat{p}_{-1}) - D\partial^2_\varphi \hat{p}_{-1} = V u_0|_{y=0}.\label{eqn-for-hat-p_-1}
\end{equation}
This equation, coupled with the equation for the bulk distribution $u_0$, forms the limiting system; after identifying $u=u_0$, $w=p_{-1}$ the limiting system is exactly~\eqref{eqn-1Dlimiting}. Though we managed to find a closed form system for $u_0$ and $p_{-1}$, to prove convergence of $f_\varepsilon$ to the combination of solutions of the limiting system, we need to find more information about $p_0$. To this end, substitute \eqref{formula-for-p--1} into the equation in the boundary-value problem \eqref{eq:p0_bvp}: 
\begin{equation}
V\partial_zp_0+\partial_z^2p_0 = (A_1+A_2z+A_3z^2)e^{-Vz},
\label{ode-for-p_0}
\end{equation}
where $A_1$, $A_2$, and $A_3$ are independent of $z$ and defined by 
\begin{eqnarray*}
A_1&=&(\partial_t+\mathcal{L}_\varphi)(\hat{p}_{-1}V)=\partial_t\left(\hat{p}_{-1}V\right)+\partial_\varphi\left(\Phi\hat{p}_{-1}V\right)-D\partial_\varphi^2\left(\hat{p}_{-1}V\right),\\
A_2&=&-V'V\Phi\hat{p}_{-1}+2DV'\partial_\varphi\left(\hat{p}_{-1}V\right)+DV''V\Phi\hat{p}_{-1}, \\ 
A_3&=& -D\left(V'\right)^2V\Phi\hat{p}_{-1}.
\end{eqnarray*}
Integration of \eqref{ode-for-p_0} with respect to $z$ together with matching condition \eqref{matching-condition} gives
\begin{equation}
p_0(t,z,\varphi) = \hat{u}_0(t,\varphi) + m_3(z,t,\varphi)e^{-zV}, \label{another-representation-for-p_0}
\end{equation}
where $m_3$ is the cubic polynomial with respect to $z$ with coefficients independent of $\varepsilon$ and depending on $t$ and $\varphi$ and $\hat{u}_0(t,\varphi)=u_0(t,\varphi,0)$.  
The representation \eqref{another-representation-for-p_0} implies
\begin{equation}
\partial_t p_0 +\partial_\varphi\left(\Phi p_0\right)-D\partial_\varphi^2 p_0 = \partial_t \hat{u}_0 +\partial_\varphi\left(\Phi\hat{u}_0\right)-D\partial_\varphi \hat{u}_0 + m_5e^{-Vz}. \label{eqn_for_p_0}
\end{equation}
Here, $m_5=m_5(z,t,\varphi)$ is the polynomial with respect to $z$ of the degree 5 with coefficients independent of $\varepsilon$ and depending on $t$ and $\varphi$. 

\subsubsection*{Composite expansion}
The composite expansion $\hat{f}_\varepsilon$ for the solution of the boundary-value problem \eqref{eq:main_fp} is given by
\begin{equation}\label{eq:composite_exp}
\hat{f}_{\varepsilon}(t,\varphi,y) := \frac{1}{\varepsilon}p_{-1}\left(t,\varphi,\frac{y}{\varepsilon}\right)+p_{0}\left(t,\varphi,\frac{y}{\varepsilon}\right)+u_{0}(t,\varphi,y)-u_{0}(t,\varphi,0),
\end{equation}
where $p_{-1}$, $p_0$, and $u_0$ are given by \eqref{formula-for-p--1}, \eqref{eqn-for-hat-p_-1}, \eqref{eqn_for_p_0}, and \eqref{eqn-for-u_0}. Next, we write the boundary-value problem for $\hat{f}_\varepsilon$. 
Specifically, the PDE for $\hat{f}_\varepsilon$ is
\begin{eqnarray}
&& \partial_t \hat{f}_\varepsilon +\partial_\varphi (\Phi \hat{f}_\varepsilon) - D \partial_\varphi^2\hat{f}_\varepsilon +\left[-V\partial_y\hat{f}_\varepsilon - \varepsilon\partial_y^2 \hat{f}_\varepsilon\right]\nonumber\\
&& = \dfrac{1}{\varepsilon}\left[\partial_t p_{-1}+\partial_\varphi\left(\Phi p_{-1}\right)-D\partial_\varphi^2 p_{-1}\right] -\dfrac{1}{\varepsilon^2}\left[V\partial_z p_{-1}+\partial_z^2 p_{-1}\right]\nonumber \\
&& \hspace{25 pt}+\left[\partial_t p_{0}+\partial_\varphi\left(\Phi p_{0}\right)-D\partial_\varphi^2 p_{0}\right] -\dfrac{1}{\varepsilon}\left[V\partial_z p_{0}+\partial_z^2 p_{0}\right]\nonumber \\
&& \hspace{25 pt}+\left[\partial_t u_0 +\partial_\varphi (\Phi u_0) - D \partial_\varphi^2u_0 -V\partial_yu_0\right] - \varepsilon\partial_y^2 u_0 \nonumber\\
&& \hspace{25pt}-\left[\partial_t \hat{u}_0 +\partial_\varphi (\Phi \hat{u}_0) - D \partial_\varphi^2\hat{u}_0\right]\nonumber\\
&&= m_5\left(\dfrac{y}{\varepsilon},t,\varphi\right)e^{-Vy/\varepsilon}-\varepsilon\partial_y^2 u_0. \nonumber 
\end{eqnarray}
Similarly, the boundary conditions for $\hat{f}_\varepsilon$ at $\{y=0\}$: 
\begin{eqnarray}
-V\hat{f}_\varepsilon-\partial_y\hat{f}_\varepsilon&=&\dfrac{1}{\varepsilon}\left[-Vp_{-1}-\partial_zp_{-1}\right]+\left[-Vp_0-\partial_zp_0\right]\nonumber \\&&+\left[-Vu_0-\varepsilon\partial_yu_0\right]+V\hat{u}_0\nonumber \\
&=& -\varepsilon\partial_y u_0 - V(u_0 - \hat{u}_0)\nonumber \\
&=& -\varepsilon\partial_y u_0. \nonumber
\end{eqnarray}	
Therefore, $\hat{f}_\varepsilon$ satisfies the following boundary-value problem 
\begin{equation}
\left\{
\begin{array}{l}
\partial_t \hat{f}_\varepsilon +\partial_\varphi (\Phi \hat{f}_\varepsilon) - D \partial_\varphi^2\hat{f}_\varepsilon +\left[-V\partial_y\hat{f}_\varepsilon - \varepsilon\partial_y^2 \hat{f}_\varepsilon\right] = m_5\left(\dfrac{y}{\varepsilon},t,\varphi\right)e^{-Vy/\varepsilon}-\varepsilon\partial_y^2 u_0,\\
-V\hat{f_\varepsilon}-\eps\partial_y\hat{f}_\varepsilon|_{y=0}=-\varepsilon \partial_y u_0|_{y=0}.
\end{array}
\right.\nonumber
\end{equation}
Note that if we denote right-hand sides of the boundary-value problem above by $R_\varepsilon$ and $r_\varepsilon$, that is, 
\begin{equation}\label{def-of-residuals}
R_\varepsilon: = m_5\left(\dfrac{y}{\varepsilon},t,\varphi\right)e^{-Vy/\varepsilon}-\varepsilon\partial_y^2 u_0 \text{ and }r_\varepsilon: = -\varepsilon \partial_y u_0|_{y=0},
\end{equation}
then 
\begin{equation}
\int\limits_{0}^{2\pi}\int\limits_0^{\infty} |R_\varepsilon| \,\text{d}y\text{d}\varphi + \int\limits_0^{2\pi} |r_\varepsilon|\,\text{d}\varphi = O(\varepsilon).
\label{residual_estimates}
\end{equation}

\subsubsection*{$L^1$-convergence of composite expansion}
Let $E_\eps(t, \varphi, y)$ be the difference between the pre-limiting solution $f_\varepsilon$ and composite expansion $\hat{f}_\varepsilon$:
\begin{equation}
E_\eps := f_\eps - \hat{f}_\eps.
\end{equation}
$E_\eps$ satisfies the following boundary-value problem:
\begin{equation} \label{eq:error_system}
\begin{cases}
    \partial_t E_\eps + \partial_\varphi(\Phi E_\eps) - D\partial^2_\varphi E_\eps - V\partial_y E_\eps - \eps\partial_y^2 E_\eps = R_\eps, \\
    [-V E_\eps - \eps\partial_y E_\eps]|_{y=0} = r_\eps,
\end{cases}
\end{equation}
where $R_\eps$ and $r_\eps$ are from \eqref{def-of-residuals}. Next, we obtain an $L^1$ {\it a priori} estimate for $E_\varepsilon$. Note that $E_\varepsilon$ is not sign-definite and thus such an estimate cannot be obtained by simply integrating the equation for $E_\varepsilon$. Instead, we will use the multiplier which is an approximation of $\sgn(E_\varepsilon)$. Namely, introduce 
$$J_\eta(s) := \sqrt{s^2 + \eta^2}-\eta$$ for a small parameter $\eta > 0$. Note the following properties of $J_\eta(s)$ as $\eta \to 0^+$:
\begin{itemize}
    \item $J_\eta(s) \to |s|$,
    \item $J'_\eta(s) = \frac{s}{\sqrt{s^2+\eta^2}} \to \sgn(s)$ for $s \neq 0$,
    \item $J''_\eta(s) = \frac{\eta^2}{(s^2+\eta^2)^{3/2}} > 0$.
\end{itemize}
The term $'-\eta'$ in the definition of $J_\eta$ is needed because $y$ ranges on an unbounded interval and we need $J_\eta(E_\eps)\in L^1$. Indeed, due to that $J_\eta(s)< |s|$ and $f_\eps,\hat{f}_\eps$ are from $L^1$ we have that $J_\eta(E_\eps)$ is from $L^1$. 

\noindent Next, multiplying by $J'(E_\eps)$ and integrating the equation for $E_\varepsilon$ over $\varphi$ and $y$, we get:
\begin{eqnarray}
    \frac{\mathrm{d}}{\mathrm{d}t}\left[ \int\limits_0^{2\pi}\int\limits_0^{\infty} J_\eta(E_\eps) \,\mathrm{d}y\mathrm{d}\varphi\right] &=& \int\limits_0^{2\pi}\int\limits_0^{\infty}  J'_\eta(E_\eps) \left( -\partial_\varphi(\Phi E_\eps) + D\partial^2_\varphi E_\eps  \right) \,\mathrm{d}y\mathrm{d}\varphi \nonumber\\
    &&+ \int\limits_0^{2\pi}\int\limits_0^{\infty}J'_\eta(E_\eps) \left( V\partial_y E_\eps +\eps\partial_y^2 E_\eps \right)\,\mathrm{d}y\mathrm{d}\varphi \nonumber \\
    &&+ \int\limits_0^{2\pi}\int\limits_0^{\infty}J'_\eta(E_\eps) R_\eps \,\,\mathrm{d}y\mathrm{d}\varphi \nonumber\\
    &=& \text{I}_1+\text{I}_2+\text{I}_3.\label{eq:error_identity}
\end{eqnarray}
Using integration by parts we can rewrite the first integral, $\text{I}_1$,  in the right-hand side of \eqref{eq:error_identity} as follows: 
\begin{eqnarray*}
\text{I}_1=\int\limits_0^{2\pi}\int\limits_0^{\infty}\partial_\varphi \Phi \left[J_\eta(E_\eps)-J'_\eta(E_\eps)E_\eps\right]\,\mathrm{d}y\mathrm{d}\varphi-D\int\limits_0^{2\pi}\int\limits_0^{\infty}J''_\eta(E_\eps) \left(\partial_\varphi E_\eps\right)^2\,\mathrm{d}y\mathrm{d}\varphi.
\end{eqnarray*}
Note that 
\begin{equation*}
J_\eta(E_\eps)-J'_\eta(E_\eps)E_\eps = \dfrac{\eta E_\eps^2}{\sqrt{\eta^2+E_\eps^2}\,(\eta+\sqrt{\eta^2+E_\eps^2})}< |E_\eps|.
\end{equation*}
Therefore, 
\begin{equation*}
\text{I}_1 \leq \|\Phi'\|_{L^{\infty}} \int\limits_0^{2\pi}\int\limits_0^{\infty} |E_\eps|\,\mathrm{d}y\mathrm{d}\varphi.
\end{equation*}
Similarly, we rewrite $\text{I}_2$: 
\begin{eqnarray*}
\text{I}_2&=&-\int\limits_0^{2\pi}\int\limits_0^{\infty}\left.\left[J_\eta(E_\eps)V+\varepsilon J'_\eta(E_\eps)\partial_yE_\eps\right]\right|_{y=0}\,\mathrm{d}\varphi-\varepsilon\int\limits_0^{2\pi}\int\limits_0^{\infty}J''_\eta(E_\eps) \left(\partial_y E_\eps\right)^2\,\mathrm{d}y\mathrm{d}\varphi 
\end{eqnarray*}
Due to boundary conditions for $E_\eps$ in \eqref{eq:error_system} we get 
\begin{equation*}
\left.\left[J_\eta(E_\eps)V+\varepsilon J'_\eta(E_\eps)\partial_yE_\eps\right]\right|_{y=0}=\left.\left[V\left\{J_\eta(E_\eps)-J'_\eta(E_\eps)E_\eps\right\}-r_\eps J'(E_\eps)\right]\right|_{y=0}
\end{equation*}
Now, use $J_\eta(E_\eps)-J'_\eta(E_\eps)E_\eps\leq \eta$ to obtain 
\begin{equation*}
\text{I}_2 \leq 2\pi\|V\|_{L^{\infty}}\eta +\int\limits_0^{2\pi}|r_\eps|\,\text{d}\varphi.
\end{equation*}
Collecting estimates for $\text{I}_1$ and $\text{I}_2$ and substituting them into \eqref{eq:error_identity} we get 
\begin{eqnarray}
&&\frac{\mathrm{d}}{\mathrm{d}t}\left[ \int\limits_0^{2\pi}\int\limits_0^{\infty} J_\eta(E_\eps) \,\mathrm{d}y\mathrm{d}\varphi\right]\leq\|\Phi'\|_{L^{\infty}} \int\limits_0^{2\pi}\int\limits_0^{\infty} J_\eta(E_\eps)\,\mathrm{d}y\mathrm{d}\varphi\nonumber \\ && \hspace{90pt}+2\pi\|V\|_{L^{\infty}}\eta +\int\limits_0^{2\pi}|r_\eps|\,\text{d}\varphi+\int\limits_0^{2\pi}\int\limits_0^{\infty} |R_\eps|\,\mathrm{d}y\mathrm{d}\varphi\nonumber \\
 && \hspace{90pt}+\|\Phi'\|_{L^{\infty}} \int\limits_0^{2\pi}\int\limits_0^{\infty} |J_\eta(E_\eps)-|E_\eps||\,\mathrm{d}y\mathrm{d}\varphi.\label{almost-final-inequality}
\end{eqnarray}
 By integrating this inequality and passing to the limit $\eta\to 0$ we get the desired convergence for $0\leq t \leq T$ for some constant $C_T>0$:
 \begin{equation}
\int\limits_0^{2\pi}\int\limits_0^{\infty}|E_\eps(t,y,\varphi)|\,\mathrm{d}y\mathrm{d}\varphi \leq C_T \int\limits_0^T\left(\int\limits_0^{2\pi}|r_\eps|\,\text{d}\varphi+\int\limits_0^{2\pi}\int\limits_0^{\infty} |R_\eps|\,\mathrm{d}y\mathrm{d}\varphi\right)\,\text{d}t.\label{estimate-l1}
 \end{equation}
 Note that 
\begin{equation}
\int\limits_0^{2\pi}\int\limits_0^{\infty}|J_{\eta}(E_\eps) - |E_\eps|| \,\mathrm{d}y\mathrm{d}\varphi =  \int\limits_0^{2\pi}\int\limits_0^{\infty} \dfrac{2\eta|E_\eps|}{\sqrt{E_\eps^2+\eta^2}+\eta+|E_\eps|} \,\mathrm{d}y\mathrm{d}\varphi.
\end{equation}
and this integral (and similarly, the integral with integration over $0\leq t\leq T$) vanishes as $\eta\to 0$ by Dominated Convergence Theorem. Combination of \eqref{estimate-l1} and \eqref{residual_estimates} concludes the proof of Theorem~\ref{thm:l1convergence}. 




\section{Proof of Theorem~\ref{thm:weakL2}: weak $L^2$-convergence for $D=0$}\label{sec:l2}

Let \( f_\varepsilon \) be a solution of \eqref{eq:main_fp-noD}. Recall that \( m_\varepsilon \) and \( u_\varepsilon \) are defined via the orthogonal decomposition in \eqref{eq:forthogonal}. We divide the proof into two parts. First, in Lemma~\ref{prop:gronwall}, we establish \textit{a priori} estimates in $L^2$. Second, we use these estimates to pass to the limit \( \varepsilon \to 0 \) and recover the limiting system \eqref{eq:limiting-noD} in Lemma~\ref{prop:weak_convergence}.

Substitute the decomposition \eqref{eq:forthogonal} into the equation from \eqref{eq:main_fp-noD}:
    \begin{align}
        \label{eq:main_plugged_in}
            \dfrac{1}{\varepsilon} \partial_t m_\varepsilon e^{-Vy/\varepsilon} &+ \dfrac{1}{\varepsilon}\partial_\varphi\left(\Phi m_\varepsilon e^{-Vy/\varepsilon}\right)\nonumber \\
            &\hspace{-20pt}+ \partial_t u_\varepsilon + \partial_\varphi\left(\Phi u_\varepsilon\right) -V\partial_y u_\varepsilon -\varepsilon \partial_y^2 u_\varepsilon = 0.
    \end{align}
    and $u_\varepsilon$ satisfies no-flux boundary conditions
    \begin{equation}
        -V u_\varepsilon - \varepsilon \partial_y u_\varepsilon \vert_{y=0} = 0.\label{eq:no_flux_u}
    \end{equation}

    \begin{lemma}[{\it a priori} estimates on $u_\varepsilon$ and $m_\varepsilon$] \label{prop:gronwall}Let $m_\varepsilon, u_\varepsilon$ be solutions of \eqref{eq:main_plugged_in}-\eqref{eq:no_flux_u}.  Then there exists a constant $C_T$ depending only on $T$ such that 
        \begin{equation}
\int\limits_{0}^{\infty}\int\limits_{0}^{2\pi} u_\varepsilon^2\, \text{d}\varphi\mathrm{d}y+\int\limits_0^{2\pi} \frac{m_\varepsilon^2}{4V}\,\mathrm{d}\varphi+\varepsilon\int\limits_0^{T}\int\limits_{0}^{\infty}\int\limits_{0}^{2\pi} (\partial_y u_\varepsilon)^2 \,\mathrm{d}\varphi\mathrm{d}y\mathrm{d}t<C_T.
    \label{eq:gronwall_no_diffusion}
\end{equation}
    \end{lemma}
\begin{proof}

\noindent\underline{STEP 1.} {\it Estimate for $u_\varepsilon$.}

\noindent Multiplying \eqref{eq:main_plugged_in} by $u_\varepsilon$ and integrating in $\int_0^{\infty}\int_0^{2\pi} \cdot \, \text{d}\varphi\text{d}y$:
        \begin{equation}      
        \begin{aligned}
            \dfrac{1}{2} \frac{\text{d}}{\text{d}t} \left[\int\limits_{0}^{\infty}\int\limits_{0}^{2\pi} u_\varepsilon^2\, \text{d}\varphi\text{d}y\right] &= \mathop{\underbrace{-\int\limits_{0}^{\infty}\int\limits_{0}^{2\pi} \partial_\varphi(\Phi u_\varepsilon) u_\varepsilon\, \text{d}\varphi\text{d}y}}_{=\text{I}_1} \\
            &+\mathop{\underbrace{\int\limits_{0}^{\infty}\int\limits_{0}^{2\pi} \partial_y\left((V+\varepsilon \partial_y) u_\varepsilon \right) u_\varepsilon \, \text{d}\varphi\text{d}y}}_{\text{I}_2}\\
            &\mathop{\underbrace{- \dfrac{1}{\varepsilon}\int\limits_{0}^{\infty}\int\limits_{0}^{2\pi} \partial_\varphi \left(\Phi m_\varepsilon e^{-Vy/\varepsilon}\right)u_\varepsilon\, \text{d}\varphi\text{d}y}}_{\mathop{\text{I}_3}}.
        \end{aligned}  \label{eq:initial_identity_for_u}
    \end{equation}

We estimate each of the integrals:
\begin{eqnarray}
\text{I}_1&=&\int\limits_{0}^{\infty}\int\limits_{0}^{2\pi} \Phi u_\varepsilon \partial_\varphi u_\varepsilon\, \text{d}\varphi\text{d}y\nonumber \\
    &=&\dfrac{1}{2}\int\limits_{0}^{\infty}\int\limits_{0}^{2\pi} \Phi \partial_\varphi(u_\varepsilon^2)\, \text{d}\varphi\text{d}y\nonumber \\
    &=& -\dfrac{1}{2}\int\limits_{0}^{\infty}\int\limits_{0}^{2\pi}\Phi^\prime \,u_\varepsilon^2 \,\text{d}\varphi \text{d}y. \label{eq:estimate_for_I_1}
\end{eqnarray}

\begin{eqnarray*}
\text{I}_2&=&\int\limits_{0}^{\infty}\int\limits_{0}^{2\pi} \partial_y\left((V+\varepsilon \partial_y) u_\varepsilon \right) u_\varepsilon \, \text{d}\varphi\text{d}y\nonumber\\
    &=&-\int\limits_{0}^{\infty}\int\limits_{0}^{2\pi}Vu_\varepsilon\partial_y u_\varepsilon \, \text{d}\varphi\text{d}y - \varepsilon \int\limits_{0}^{\infty}\int\limits_{0}^{2\pi}(\partial_y u_\varepsilon)^2 \,\text{d}\varphi\text{d}y\nonumber\\
    &=&-\dfrac{1}{2}\int\limits_0^{2\pi} V u_\varepsilon^2|_{y=0}\,\text{d}\varphi  - \varepsilon \int\limits_{0}^{\infty}\int\limits_{0}^{2\pi}(\partial_y u_\varepsilon)^2 \,\text{d}\varphi\text{d}y\nonumber\\
    &=&-\dfrac{1}{2}\int\limits_0^{2\pi} V \left(\left.u_\varepsilon e^{-Vy/\varepsilon}\right|_{y=0}\right)^2\,\text{d}\varphi  - \varepsilon \int\limits_{0}^{\infty}\int\limits_{0}^{2\pi}(\partial_y u_\varepsilon)^2 \,\text{d}\varphi\text{d}y\nonumber\\
     &=&-\dfrac{1}{2}\int\limits_0^{2\pi} V \left(\int\limits_0^{\infty}e^{-Vy/\varepsilon}\partial_yu_\varepsilon\,\text{d}y\right)^2\,\text{d}\varphi  - \varepsilon \int\limits_{0}^{\infty}\int\limits_{0}^{2\pi}(\partial_y u_\varepsilon)^2 \,\text{d}\varphi\text{d}y
\end{eqnarray*}

We then apply Cauchy-Schwarz to the first integral in $y$ and obtain
    \begin{equation}
     \text{I}_2\leq-\dfrac{3\varepsilon}{4}\int\limits_{0}^{\infty}\int\limits_{0}^{2\pi}(\partial_y u_\varepsilon)^2 \,\text{d}\varphi\text{d}y.\label{eq:estimate_for_I_2}
    \end{equation}

\begin{eqnarray}
    \text{I}_3 &=& -\dfrac{1}{\varepsilon}\int\limits_0^\infty \int\limits_0^{2\pi} \Phi^\prime m_\varepsilon e^{-Vy/\varepsilon} u_\varepsilon \,\text{d}\varphi\text{d}y -\dfrac{1}{\varepsilon}\int\limits_0^\infty \int\limits_0^{2\pi} \Phi \partial_\varphi m_\varepsilon e^{-Vy/\varepsilon} u_\varepsilon \,\text{d}\varphi\text{d}y \nonumber\\
    &&\hspace{20pt}+\, \dfrac{1}{\varepsilon^2}\int\limits_0^\infty \int\limits_0^{2\pi} \Phi V^\prime  m_\varepsilon e^{-Vy/\varepsilon} y u_\varepsilon \,\text{d}\varphi\text{d}y
\end{eqnarray}

Note that by orthogonality condition \eqref{eq:forthogonal}, the first two terms are zero.  Then to estimate the remaining term, since $\Phi, V,$ and $m_\varepsilon$ do not depend on $y$, the integral can be rewritten as
\begin{eqnarray}
     \text{I}_3 &=&\dfrac{1}{\varepsilon^2}\left(\int\limits_0^{2\pi} \Phi V^\prime  m_\varepsilon\,\text{d}\varphi\right)\left( \int\limits_0^\infty e^{-Vy/\varepsilon} y u_\varepsilon \,\text{d}y\right)
\end{eqnarray}

To estimate the integral in $y$, note
\begin{eqnarray}
&&\int\limits_{0}^{\infty}ye^{-Vy/\varepsilon}u_\varepsilon\,\text{d}y=\dfrac{\varepsilon}{V}\int\limits_{0}^{\infty}ye^{-Vy/\varepsilon}\partial_y u_\varepsilon\,\text{d}y\nonumber\\
&&\hspace{20pt}\leq\dfrac{\varepsilon}{V}\left(\int\limits_0^{\infty}y^2e^{-2Vy/\varepsilon}\,\text{d}y\right)^{1/2}\left(\int\limits_0^{\infty}(\partial_yu_\varepsilon)^2\,\text{d}y\right)^{1/2}\nonumber\\
&&\hspace{20pt}=\dfrac{\varepsilon^{5/2}}{2V^{5/2}}\left(\int\limits_0^{\infty}(\partial_yu_\varepsilon)^2\,\text{d}y\right)^{1/2}
\end{eqnarray}      
Thus we have 
\begin{eqnarray}
\text{I}_3&\leq& \varepsilon^{1/2}\int\limits_0^{2\pi}\dfrac{m_\varepsilon}{V^{5/2}} \left(\int\limits_0^{\infty}(\partial_yu_\varepsilon)^2\,\text{d}y\right)^{1/2} \,\text{d}\varphi\nonumber 
\\&\leq&\dfrac{\varepsilon}{4}\int\limits_0^{2\pi}\int\limits_0^{\infty}(\partial_yu_\varepsilon)^2\,\text{d}y\text{d}\varphi+
C\int\limits_0^{2\pi}\dfrac{m_\varepsilon^2}{4V}\,\text{d}\varphi. \label{eq:estimate_for_I_3}
\end{eqnarray}

Combining \eqref{eq:initial_identity_for_u},\eqref{eq:estimate_for_I_1}, \eqref{eq:estimate_for_I_2}, and \eqref{eq:estimate_for_I_3} we get
\begin{equation}
       \frac{\text{d}}{\text{d} t} \left[\int\limits_{0}^{\infty}\int\limits_{0}^{2\pi} u_\varepsilon^2\, \text{d}\varphi\text{d}y\right] +\varepsilon \int\limits_{0}^{\infty}\int\limits_{0}^{2\pi} (\partial_y u_\varepsilon)^2 \,\text{d}\varphi\text{d}y\leq 2\int\limits_{0}^{\infty}\int\limits_{0}^{2\pi} u_\varepsilon^2\, \text{d}\varphi\text{d}y + C \int\limits_0^{2\pi}\dfrac{m_\varepsilon^2}{4V}\,\text{d}\varphi.\label{eq:main_est_for_u}
 \end{equation}

\newpage

\noindent\underline{STEP 2.} {\it Estimate for $m_\varepsilon$.}

\noindent Multiply \eqref{eq:main_plugged_in} by $m_\eps e^{-Vy/\varepsilon}$, integrate $\int_0^\infty\int_0^{2\pi} \cdot \, \text{d}y\text{d}\varphi$ and use $\int_0^{\infty}e^{-2Vy/\eps}\,\text{d}y=\frac{\eps}{2V}$:
 	\begin{eqnarray}
            &&\frac{\text{d}}{\text{d}t} \left[\int\limits_0^{2\pi} \dfrac{m_\eps^2}{4V}   \,\text{d}\varphi\right] + \dfrac{1}{\varepsilon}\int\limits_0^{2\pi}\int\limits_0^{\infty} \partial_\varphi\left(\Phi m_\varepsilon e^{-Vy/\varepsilon}\right)m_\varepsilon e^{-Vy/\varepsilon}\,\text{d}y\text{d}\varphi\nonumber \\
            && \hspace{80pt}+ \int\limits_0^{\infty}\int\limits_{0}^{2\pi} \partial_y \left((-V-\eps\partial_y) u_\eps \right) m_\eps e^{-Vy/\eps}\,\text{d}\varphi\text{d}y = 0. \label{energy_m_eps}
       \end{eqnarray}
    We rewrite the second term above as follows: 
   \begin{eqnarray*}
   &&\dfrac{1}{\varepsilon}\int\limits_0^{2\pi}\int\limits_0^{\infty} \partial_\varphi\left(\Phi m_\varepsilon e^{-Vy/\varepsilon}\right)m_\varepsilon e^{-Vy/\varepsilon}\,\text{d}y\text{d}\varphi \\
   &&\hspace{80pt}=\int\limits_0^{2\pi}\dfrac{\Phi'm_\varepsilon^2}{2V}\,\text{d}\varphi+\dfrac{1}{2}\int\limits_0^{\infty}\int\limits_0^{2\pi}\Phi \partial_\varphi\left[m_\varepsilon^2 e^{-2Vy/\varepsilon}\right]\,\text{d}\varphi\text{d}y\\
    &&\hspace{80pt}=\int\limits_0^{2\pi}\dfrac{\Phi'm_\varepsilon^2}{4V}\,\text{d}\varphi.
   \end{eqnarray*}
    We put the third term in \eqref{energy_m_eps} into the right-hand side and estimate as follows using integration by parts and orthogonality condition from \eqref{eq:forthogonal}:
    \begin{eqnarray*}
    && \int\limits_0^{\infty}\int\limits_{0}^{2\pi} \partial_y \left((V+\eps\partial_y) u_\eps \right) m_\eps e^{-Vy/\eps}\,\text{d}\varphi\text{d}y \\
    && \hspace{80pt} = \int\limits_0^{\infty}\int\limits_{0}^{2\pi}V \partial_y u_\eps \, m_\eps e^{-Vy/\eps} \,\text{d}\varphi\text{d}y\\
    && \hspace{80pt} \leq \int\limits_0^{2\pi}Vm_\eps\left(\varepsilon\int\limits_0^{\infty}(\partial_y u_\eps)^2\,\text{d}y\right)^{1/2}\left(\dfrac{1}{\varepsilon}\int\limits_0^{\infty}e^{-2Vy/\eps}\,\text{d}y\right)^{1/2}\,\text{d}\varphi \\
    && \hspace{80pt} \leq 2\|V\|^2_{L^{\infty}}\int\limits_0^{2\pi}\dfrac{m_\eps^2}{4V}\,\text{d}\varphi+\dfrac{\eps}{4}\int\limits_0^{2\pi}\int\limits_0^{\infty}(\partial_y u_\eps)^2\,\text{d}y\text{d}\varphi.
    \end{eqnarray*}
    
    \noindent Now we can rewrite \eqref{energy_m_eps} as follows: 
    \begin{equation}
     \begin{aligned}
     \frac{\text{d}}{\text{d}t} \left[\int\limits_0^{2\pi} \frac{m_\eps^2}{4V}\,\text{d}\varphi \right]&\leq C\int\limits_0^{2\pi}\dfrac{m_\eps^2}{4V}\,\text{d}\varphi+\dfrac{\eps}{4}\int\limits_0^{2\pi}\int\limits_0^{\infty}(\partial_y u_\eps)^2\,\text{d}y\text{d}\varphi.
    \end{aligned}\label{final-estimate-m}
    \end{equation}
    Here, $C:=\|\Phi'\|_{L^{\infty}}+2\|V\|_{L^{\infty}}>0$.

    \medskip 

    \newpage

    \noindent\underline{STEP 3.} {\it Concluding the proof of Lemma~\ref{prop:gronwall}}
    
    Adding estimates \eqref{final-estimate-m} and \eqref{eq:main_est_for_u} we get
    \begin{eqnarray*}
      &&\frac{\text{d}}{\text{d}t} \left[\int\limits_{0}^{\infty}\int\limits_{0}^{2\pi} u_\eps^2\, \text{d}\varphi\text{d}y+\int\limits_0^{2\pi} \frac{m_\eps^2}{4V}\,\text{d}\varphi\right]+\dfrac{3\eps}{4} \int\limits_{0}^{\infty}\int\limits_{0}^{2\pi} (\partial_y u_\eps)^2 \,\text{d}\varphi\text{d}y\\
      &&\hspace{120pt}\leq C\left(\int\limits_{0}^{\infty}\int\limits_{0}^{2\pi} u_\eps^2\, \text{d}\varphi\text{d}y+\int\limits_0^{2\pi} \frac{m_\eps^2}{4V}\,\text{d}\varphi\right).
    \end{eqnarray*} 
    Applying Gronwall estimate gives \eqref{eq:gronwall_no_diffusion}, and it concludes the proof of Lemma~\ref{prop:gronwall}.
 \end{proof}

    \begin{lemma}[passing to the limit $\varepsilon \to 0$]\label{prop:weak_convergence}
        Let $m_\varepsilon, u_\varepsilon$ be solutions of \eqref{eq:main_plugged_in}.  Then
        \begin{align}
            m_\varepsilon &\rightharpoonup m \text{ in $L^2([0,2\pi))$}\\
            u_\varepsilon &\rightharpoonup u \text{ in $H^1((0,\infty)) \times L^2([0,2\pi)$.}
        \end{align}

        where $m,u$ satisfy the limiting equations:
        \begin{equation}
\begin{cases}
    \partial_t u  + \partial_\varphi(\Phi u) - V\partial_y u = 0,\\
    \partial_t\left(\dfrac{m}{V}\right) + \partial_\varphi \left(\dfrac{\Phi m}{V}\right) = -V\cdot u \vert_{y=0}.
\end{cases}
    \label{eq:limit_no_diffusion}
\end{equation}
    \end{lemma}
    \begin{proof}
    Multiply \eqref{eq:main_plugged_in} by $v \in C^\infty([0,T]\times[0,\infty)\times[0,2\pi))$, $v|_{t=0}=v|_{t=T}=0$ and $v$ is periodic in $\varphi$, integrate over $t,y,\varphi$, and integrate by parts. We obtain the definition of a weak solution:
    \begin{eqnarray}
            &&-\frac{1}{\eps} \int\limits_0^T\int\limits_0^\infty\int\limits_0^{2\pi} m_\eps e^{-Vy/\eps}v_t\,\text{d}\varphi\text{d}y\text{d}t -\frac{1}{\eps} \int\limits_0^T\int\limits_0^\infty\int\limits_0^{2\pi}\Phi m_\eps e^{-Vy/\eps} \partial_\varphi v\,\text{d}\varphi\text{d}y\text{d}t \nonumber\\
            &&\hspace{30pt}-\int\limits_0^T\int\limits_0^\infty\int\limits_0^{2\pi} u_\eps \partial_t v\,\text{d}\varphi\text{d}y\text{d}t  +\int\limits_0^T\int\limits_0^\infty\int\limits_0^{2\pi}u_\eps (V - \eps \partial_y) \partial_yv \,\text{d}\varphi\text{d}y\text{d}t \nonumber\\
            &&\hspace{30pt}- \int\limits_0^T\int\limits_0^\infty\int\limits_0^{2\pi}  \Phi u_\eps \partial_\varphi v\,\text{d}\varphi\text{d}y\text{d}t-\eps\int\limits_0^T \int\limits_0^{2\pi} u_\eps|_{y=0} \cdot \partial_y v|_{y=0}\,\text{d}\varphi\text{d}t =0. 
         \label{eq:weakform}
    \end{eqnarray}
    Set $v(t,y,\varphi) = h(t,\varphi) e^{-Vy/\eps}$.
    \begin{eqnarray}
           &&-\int\limits_0^T\int\limits_0^{2\pi} \dfrac{m_\eps}{2V} h_t\,\text{d}\varphi\text{d}t - \frac{1}{\eps}  \int\limits_0^T\int\limits_0^\infty\int\limits_0^{2\pi} \Phi m_\eps e^{-2Vy/\eps} \partial_\varphi h\,\text{d}\varphi\text{d}y\text{d}t\nonumber\\
            &&\hspace{80pt}- \frac{1}{\eps^2}\int\limits_0^T\int\limits_0^\infty\int\limits_0^{2\pi} V'\Phi m_\eps ye^{-2Vy/\eps}h  \,\text{d}\varphi\text{d}y\text{d}t \nonumber\\
            &&\hspace{80pt}-\dfrac{1}{\eps}\int\limits_0^T\int\limits_0^\infty\int\limits_0^{2\pi}V'\Phi u_\eps y e^{-Vy/\eps}h\,\text{d}\varphi\text{d}y\text{d}t
            \nonumber\\
            &&\hspace{80pt}+ \int\limits_0^T \int\limits_0^{2\pi} Vh \cdot u_\eps |_{y=0}\,\text{d}\varphi\text{d}t = 0.  \label{eq:weakformphi}
    \end{eqnarray}
   Use $\int_0^{\infty}e^{-2Vy/\eps}\,\text{d}y=\eps/(2V)$ and $\int_0^{\infty}ye^{-2Vy/\eps}\,\text{d}y=\eps^2/(4V^2)$ to get 
    \begin{eqnarray}
       && -\int\limits_0^T\int\limits_0^{2\pi} \dfrac{m_\eps}{2V} h_t\,\text{d}\varphi\text{d}t -\int\limits_0^T\int\limits_0^{2\pi} \Phi\dfrac{m_\eps}{2V}  \partial_\varphi h\,\text{d}\varphi\text{d}t\nonumber\\
         &&\hspace{80pt}-\int\limits_0^T\int\limits_0^{2\pi}  \frac{V'}{{(2V)}^2}\Phi m_\eps h\,\text{d}\varphi\text{d}t\nonumber\\
          &&\hspace{80pt}-\dfrac{1}{\eps}\int\limits_0^T\int\limits_0^{2\pi}\int\limits_0^{\infty}V'\Phi u_\eps y e^{-Vy/\eps}h\,\text{d}y\text{d}\varphi\text{d}t
         \nonumber\\
                     &&\hspace{80pt}+ \int\limits_0^T \int\limits_0^{2\pi} Vh \cdot u_\eps |_{y=0}\,\text{d}\varphi\text{d}t = 0. 
         \label{eq:weakformlimit}
    \end{eqnarray}
    From Lemma~\ref{prop:gronwall} we can introduce $m(t,\varphi)$, the weak limit of $m_\eps(t,\varphi)$ as $\eps$ goes to $0$ in $L^2(0,T;L^2(0,2\pi))$, and pass to the limit in first three terms above. Here is what happens with the fourth term:
    \begin{eqnarray}
    &&-\dfrac{1}{\eps}\int\limits_0^T\int\limits_0^{2\pi}\int\limits_0^{\infty}V'\Phi u_\eps y e^{-Vy/\eps}h\,\text{d}y\text{d}\varphi\text{d}t\nonumber\\
    &&\hspace{40pt}\leq \dfrac{2\|h\|_{\infty}}{\eps}\left(\int\limits_0^{T}\int\limits_0^{2\pi}\int\limits_0^{\infty}y^2e^{-2Vy/\eps}\text{d}y\text{d}\varphi\text{d}t\right)^{1/2}\left(\int\limits_0^{T}\int\limits_0^{2\pi}\int\limits_0^{\infty}u_\eps^2\,\text{d}y\text{d}\varphi\text{d}t\right)^{1/2} \nonumber \\
    && \hspace{40 pt}\leq C\|h\|_{\infty}\sqrt{2\pi} T\eps^{1/2}. 
    \end{eqnarray}
    Thus, the fourth term vanishes as $\eps$ goes to $0$. 
    
    To pass to the limit in the last term in \eqref{eq:weakformphi}, observe
    \begin{eqnarray}
           \int\limits_0^T\int\limits_0^{2\pi} {\left(u_\eps|_{y=0}\right)}^2 \,\text{d}\varphi \text{d}t&=&  \int\limits_0^T\int\limits_0^{2\pi} {\left(u_\eps e^{-Vy/\eps}|_{y=0}\right)}^2 \,\text{d}\varphi \text{d}t\nonumber\\
           &=& \int\limits_0^T\int\limits_0^{2\pi}{\left(\int\limits_0^\infty \partial_y u_\eps e^{-Vy/\eps} \, dy\right)}^2 \text{d}\varphi \text{d}t\nonumber\\
           &\leq&\int\limits_0^T\int\limits_0^{2\pi}\left(\int\limits_0^\infty (\partial_y u_\eps)^2\,\text{d}y\right)\left(\int\limits_0^{\infty}e^{-2Vy/\eps} \, \text{d}y \right)\text{d}\varphi \text{d}t\nonumber\\          
            &\leq& \int\limits_0^T\int\limits_0^{2\pi} \dfrac{\eps}{V}\int\limits_0^\infty  {\left(\partial_y u_\eps\right)}^2 \, \text{d}y \text{d}\varphi \text{d}t < C.
      \end{eqnarray}
    From this estimate we conclude that $u_\eps|_{y=0}\rightharpoonup u_0$ in $L^2(0,T;L^2(0,2\pi))$.  Passing to the limit in \eqref{eq:weakformlimit}, we obtain 
    \begin{equation}
        \label{eq:m_limit}
       -\int\limits_{0}^{T}\int\limits_0^{2\pi}\dfrac{m}{V}\left[h_t+\Phi\partial_\varphi h+\dfrac{V'}{2V}\Phi h\right]\,\text{d}\varphi\text{d}t=-2\int\limits_0^{T}\int\limits_0^{2\pi}Vu_0h\,\text{d}\varphi\text{d}t.
    \end{equation}
    
    \bigskip 
    
    Next, consider $v(t,y,\varphi)$, independent of $\eps$, in \eqref{eq:weakform}. We note that  
       \begin{equation}
            \lim_{\eps \to 0}\frac{1}{\eps} \int\limits_0^{T}\int\limits_{0}^{2\pi}\int\limits_0^\infty  e^{-Vy/\eps} m_\eps(t,\varphi) v(t,y,\varphi) \, \text{d}y\text{d}\varphi\text{d}t = \int\limits_0^{T}\int\limits_{0}^{2\pi}\frac{m(t,\varphi)}{V}v(t,0,\varphi)\text{d}\varphi\text{d}t.\nonumber 
        \end{equation}

    We apply this equality in \eqref{eq:weakform} with \eqref{eq:gronwall_no_diffusion} to pass to the limit in \eqref{eq:weakform}. We also use \eqref{eq:gronwall_no_diffusion} which implies weak convergence $u_\eps \rightharpoonup u$ in $L^2((0,T)\times(0,2\pi)\times (0,\infty))$  as well as established above convergence  $u_\eps|_{y=0}\rightharpoonup u_0$ in $L^2(0,T;L^2(0,2\pi))$ (note that we don't impose $u|_{y=0}=u_0$). We obtain
    \begin{eqnarray}
       &&  -\int\limits_0^T\int\limits_0^{2\pi}\frac{m}{V} \left[ v_t(t,0,\varphi) + \Phi \partial_\varphi v(t,0,\varphi)\right]\text{d}\varphi\text{d}t \nonumber\\
       && \hspace{80pt}-\int\limits_0^T\int\limits_0^{2\pi} \int\limits_0^\infty u \left[v_t-V\partial_y v+\Phi\partial_\varphi v\right]\, \text{d}y\text{d}\varphi\text{d}t=0.
        \label{eq:weakformlimit2}
    \end{eqnarray}
    Comparing this with \eqref{eq:m_limit} with $h(t,\varphi)=v(t,0,\varphi)$ we get
\begin{eqnarray}
       &&-2\int\limits_0^T \int\limits_0^{2\pi} Vu_0h\,\text{d}\varphi\text{d}t +\int\limits_0^T \int\limits_0^{2\pi} \frac{m}{2V^2} V'\Phi h\,\text{d}\varphi\text{d}t \nonumber \\
           && \hspace{80pt}-\int\limits_0^T\int\limits_0^{2\pi} \int\limits_0^\infty u \left[v_t-V\partial_y v+\Phi\partial_\varphi v\right]\, \text{d}y\text{d}\varphi\text{d}t=0.\nonumber
\end{eqnarray}
    Integrating by parts, we conclude that equation above is the weak formulation for
    \begin{align}
        &u_t + \partial_\varphi (\Phi u)-V \partial_y u   = 0,\quad y>0,0\leq\varphi <2\pi \nonumber\\
        &\label{eq:boundary_limit}-2Vu_0 + \frac{m}{2V^2}\Phi V' + u|_{y=0}V  = 0,\quad 0\leq\varphi< 2\pi. 
    \end{align}
    
    \noindent Similarly, integrating by parts \eqref{eq:m_limit} we get: 
    \begin{equation}
    \label{eq:m_limit2}\frac{1}{V}\partial_t m + \frac{m}{2{V}^2}\Phi V' + \partial_\varphi \left(\frac{\Phi}{V} m \right) = -2Vu_0,\quad 0\leq\varphi< 2\pi.
    \end{equation}
    
    \noindent Plugging \eqref{eq:boundary_limit} into \eqref{eq:m_limit2}, we get the system \eqref{eq:limit_no_diffusion}.

    \end{proof}

\section*{Acknowledgment}
The work of M.P. was supported by the Hellman Fellowship Program 2024-2025. The work of L.B. and S.D. was partially supported by the National
Science Foundation grants DMS-2404546, DMS-2005262, and PHY-2140010.
P.-E. Jabin was partially supported by NSF grants DMS-2205694 and DMS-2508570.

\bibliography{refs}

\end{document}